\documentclass[11pt]{amsart}   
\usepackage[english]{babel}
\usepackage{mathtools}
\usepackage{amssymb,amscd, graphics,color,latexsym,cancel}
\usepackage{comment}
\usepackage[linktoc=all, pagebackref, hyperindex]{hyperref}%
\hypersetup{colorlinks,  citecolor=blue,  filecolor=blue,  linkcolor=blue,  urlcolor=black}

\usepackage{tikz-cd} 
\usepackage{extpfeil}

\usepackage{tikz}
\usetikzlibrary{matrix,calc}
\usepackage{mathrsfs}
\usepackage[all]{xy}
\usepackage{comment}
\usepackage{amsfonts}
\usepackage[normalem]{ulem}

\usepackage{euscript}
\usepackage{amsmath}
\usepackage{ifpdf}
\usepackage{cleveref}
\LARGE\textwidth=6in
\newtheorem{Theorem}{Theorem}[section]
\newtheorem{Lemma}[Theorem]{Lemma}
\newtheorem{Corollary}[Theorem]{Corollary}
\newtheorem{Proposition}[Theorem]{Proposition}
\newtheorem{mthm}{Main Theorem}

\theoremstyle{definition}
\newtheorem{Remark}[Theorem]{Remark}
\newtheorem{Example}[Theorem]{Example}

\newtheorem{Question}[Theorem]{Question}
\def\sqr#1#2{{\vcenter{\hrule height.#2pt
			\hbox{\vrule width.#2pt height#1pt \kern#1pt
				\vrule width.#2pt}
			\hrule height.#2pt}}}
\def\phi{\varphi}

\def\VaVa{{\mathcal V}\kern-5pt {\mathcal V}}
\def\gr#1#2{{\rm gr}\, _{#1}(#2)}
\def\gr{{\rm gr}\,}
\def\hht{{\rm ht}\,}
\def\depth{{\rm depth}\,}

\def\Min{{\rm Min}\,}
\def\codim{{\rm codim}\,}

\def\ker{{\rm ker}\,}
\def\grade{{\rm grade}\,}
\def\rk{\rm rank}

\def\syz{\mbox{\rm Syz}}

\def\Ext#1#2#3#4{{\rm Ext}\,^{#1}_{#2}({#3},{#4})}

\def\supp#1{{\rm Supp}\, (#1)}

\def\ini{\mbox{\rm in}}

\def\cl#1{{\mathcal #1}}

\def\phi{\varphi}

\def\hht{{\rm ht}\,}
\def\grade{{\rm grade}\,}

\def\ZZ{{\bf Z}}

\def\fm{{\mathfrak m}}

\def\fm{{\mathfrak m}}

\def\cl#1{{\cal #1}}
\def\rk{\rm rank}

\newcommand{\excise}[1]{}

\def\NZQ{\mathbb}               

\def\ZZ{{\NZQ Z}}

\def\CC{{\NZQ C}}

\def\PP{{\NZQ P}}

\def\A{{\mathcal A}}

\def\G{{\mathcal G}}

\def\B{{\mathcal B}}

\def\opn#1#2{\def#1{\operatorname{#2}}} 
\opn\chara{char} \opn\length{\ell} \opn\pd{pd} \opn\rk{rk}
\opn\projdim{proj\,dim} \opn\injdim{inj\,dim} \opn\rank{rank}
\opn\depth{depth} \opn\grade{grade} \opn\height{height}
\opn\embdim{emb\,dim} \opn\codim{codim}

\opn\Tr{Tr} \opn\bigrank{big\,rank}
\opn\superheight{superheight}\opn\lcm{lcm}
\opn\trdeg{tr\,deg}
	\opn\reg{reg} \opn\lreg{lreg} \opn\ini{in} \opn\lpd{lpd}
	\opn\size{size} \opn\sdepth{sdepth}
	\opn\link{link}\opn\fdepth{fdepth}\opn\lex{lex}
	\opn\tr{tr}
	\opn\type{type}
	\opn\div{div} \opn\Div{Div} \opn\cl{cl} \opn\Cl{Cl}
	\opn\Spec{Spec} \opn\Supp{Supp} \opn\supp{supp} \opn\Sing{Sing}
	\opn\Ass{Ass} \opn\Min{Min}\opn\Mon{Mon}
	\opn\Ho{H}
	\opn\Ann{Ann} \opn\Rad{Rad} \opn\Soc{Soc}
	\opn\ann{ann}
	\opn\Im{Im} \opn\Ker{Ker} \opn\Coker{Coker} \opn\Am{Am}
	\opn\Hom{Hom} \opn\Tor{Tor} \opn\Ext{Ext} \opn\End{End}
	\opn\Aut{Aut} \opn\id{id}
	
	\opn\nat{nat}
	\opn\pff{pf}
	\opn\Pf{Pf} \opn\GL{GL} \opn\SL{SL} \opn\mod{mod} \opn\ord{ord}
	\opn\Gin{Gin} \opn\Hilb{Hilb}\opn\sort{sort}
	\opn\HS{HS}  \opn\PF{PF}\opn\Ap{Ap}\opn\HF{HF}\opn\indeg{indeg}
	\opn\aff{aff} \opn
	\con{conv} \opn\relint{relint} \opn\st{st}
	\opn\lk{lk} \opn\cn{cn} \opn\core{core} \opn\vol{vol}  \opn\inp{inp} \opn\nilpot{nilpot}
	\opn\link{link} \opn\star{star}\opn\lex{lex}\opn\set{set}
	\opn\width{wd}
	\opn\Fr{F}
	\opn\QF{QF}
	\opn\G{G}
	\opn\type{type}\opn\res{res}
	\opn\log{Log}   \opn\Der{Der}  \opn\Bour{Bour} 
    \opn\im{im}   \opn\proj{Proj}  \opn\BSpec{BSpec}
    \opn\ESpec{ESpec}
	\opn\gr{gr}
	
	\def\pot#1#2{#1[\kern-0.28ex[#2]\kern-0.28ex]}

	\opn\dirlim{\underrightarrow{\lim}}
	\opn\inivlim{\underleftarrow{\lim}}
\begin{document}
		
		\title[ Bourbaki Degree of Line Arrangements]{The Bourbaki Degree of Line Arrangements} 
		
		\author{Abbas Nasrollah Nejad}
		\address{Department of Mathematics, Institute for Advanced Studies in Basic Sciences (IASBS), Zanjan 45137-66731, Iran \newline \indent
        Instituto de Ciência e Tecnologia, Universidade Federal de S\~ao Paulo (UNIFESP),
        S\~ao Jos\'e dos Campos, SP, Brazil
        }
		\email{annejad@unifesp.br}

        \author{Samira Nouri}
		\address{Department of Mathematics, Institute for Advanced Studies in Basic Sciences (IASBS), Zanjan 45137-66731, Iran }
		\email{nouri@iasbs.ac.ir}
		

\begin{abstract}
We study the Bourbaki degree of line arrangements in $\PP^2$ and its
interaction with the intersection lattice and the syzygies of the
gradient ideal. We show that the Bourbaki degree is obtained by
evaluating the reduced characteristic polynomial at the initial degree
of the first syzygy module. This leads to a sharp upper bound and to refinements involving the two
largest intersection multiplicities, with consequences for the global
Tjurina number, the freeness defect, and Terao's conjecture, including
an improvement of Dimca's numerical criterion when the smaller
exponent is at least five. We also establish addition--deletion formulas
and show that the defect from the sharp bound is monotone under line
addition. Finally, we prove that the Bourbaki degree is determined by
the intersection lattice for arrangements with at most eight lines, while examples with isomorphic intersection lattices and different
Bourbaki degrees exist for every number of lines at least nine.
\end{abstract}
		\maketitle
\section*{Introduction}
\label{sec:introduction}

A recurring theme in the study of line arrangements is the interaction
between the geometry of the intersection points, the combinatorics of
the intersection lattice, and the homological behavior of the gradient
ideal. Line arrangements therefore form a natural meeting point of
combinatorics, commutative algebra, singularity theory, and the theory
of logarithmic derivations. The foundations of this viewpoint go back
to Saito's work on logarithmic derivations~\cite{SaitoLogarithmic}; for the arrangement-theoretic background we
refer to~\cite{Orlik-Terao}, while more recent developments related to
numerical invariants, Tjurina numbers, and derivation modules can be found in~\cite{Dimca-minimal-Tjurina, Dimca-local-derivations, Dimca-Ibadula-Macinic}.

Let $k$ be an algebraically closed field of characteristic zero and set
$R=k[x,y,z]$. Let $\A=\{L_1,\ldots,L_{n+1}\}$ be an arrangement of distinct lines in $\PP_k^2$, with defining equation
$f_\A=\ell_1\cdots\ell_{n+1}\in R$, and let $J_\A=J_{f_\A}$ be its
gradient ideal. We explicitly assume $k=\CC$ whenever required by the
results used from the literature. Two basic numerical invariants that appear throughout
this theory are the global Tjurina number $\tau(\A)$ and the initial
degree $e_\A$ of the first syzygy module $\syz(J_\A)$.

The homological point of view relates the singularities of the
arrangement to the structure of the first syzygy module of its
gradient ideal. The simplest case occurs when $\syz(J_\A)$ is free;
these are precisely the free arrangements. In this case $e_\A$ is the
smaller exponent. More generally, the degrees and the number of
generators of $\syz(J_\A)$ play an important role in the theory of
nearly free and plus-one generated arrangements. The notion of a
plus-one generated arrangement was introduced by Abe~\cite{Abe-plus-one}; related developments can be found in
\cite{Abe-Dimca-Pokora-hierarchy,
Bromboszcz-plus-one,Dimca-Ibadula-Macinic, Dimca-Kuhne-Pokora}.

The invariants $\tau(\A)$ and $e_\A$ reflect rather different aspects
of the arrangement. The Tjurina number is determined by the
intersection multiplicities, and hence by the intersection lattice,
whereas $e_\A$ need not be combinatorial. This phenomenon already
appears in Ziegler's examples~\cite{Ziegler1989} and has recently been
studied from several points of view; see
\cite{Dimca-Pokora-Jacobian-Ziegler,Dimca-Pokora-nine, Dimca-Pokora-Ziegler, DiPasquale-Sidman-Traves, Kuhne-Luber-Pokora}. The interaction between the lattice-theoretic
data and the syzygies of $J_\A$ is also closely related to Terao's
conjecture, which asks whether freeness is determined by the
intersection lattice. We refer to~\cite{Orlik-Terao} for the classical
theory and to \cite{Barakat-Kuhne} for recent computational progress.

The present paper approaches this circle of problems through the Bourbaki degree introduced by the first author, M.~Jardim, and A.~Simis in \cite{JNSBourbaki}. Further developments of this invariant in broader algebraic settings, including syzygy modules of matrices and arbitrary three-generated homogeneous ideals, can be found in~\cite{JMN-matrices,JMN-R-three-generated}. If $X=V(f)\subseteq\PP^2$ is a reduced
plane curve and $e=\indeg(\syz(J_f))$, a minimal generator
$\nu\in\syz(J_f)_e$ gives a Bourbaki sequence
\[
0\longrightarrow R(-e)\longrightarrow\syz(J_f)
\longrightarrow I_\nu(e-d)\longrightarrow0,
\]
where $\deg(f)=d+1$. The degree of the corresponding Bourbaki ideal is
independent of the choice of $\nu$ and defines the numerical invariant
$\Bour(X)$. One of the basic formulas
in \cite{JNSBourbaki} is
\[
\Bour(X)=d^2+e(e-d)-\tau(X).
\]
Thus the Bourbaki degree brings together the initial degree of the
first syzygy module and the global Tjurina number in a single
numerical invariant.

For line arrangements, this point of view becomes particularly
effective. The special form of their singularities allows the
Bourbaki degree to be expressed in terms of the reduced characteristic
polynomial, and this leads to sharp bounds involving the initial degree
and the intersection multiplicities. It also reveals a useful behavior
under addition and deletion and gives a new way to study questions of
combinatorial invariance.

Two ingredients are central to the proof of our first main result.
Proposition~\ref{prop:syzygies-associated-with-points} shows that the
syzygy associated with an intersection point $p$ of multiplicity
$m_p$ lies in $\syz(J_\A)_{n+1-m_p}$, has coordinates with no
nonconstant common factor, and is not proportional to the syzygy
associated with any other intersection point. On the Bourbaki side,
Lemma~\ref{lem:regularity-bourbaki-ideal} gives
\[
\reg(I_\nu)\leq e_\A-1.
\]
Together, these facts yield the sharp bound for the Bourbaki degree
and its refinements in terms of the intersection multiplicities.

Assume from now on that $\A$ is essential and set $e=e_\A$. Let
\[
\overline{\chi}_\A(t)=t^2-nt+b_2(\A)-n
\]
be its reduced characteristic polynomial. For a nonfree arrangement,
we also consider its \emph{type}
\[
t(\A):=\indeg(I_\nu),
\]
introduced in
\cite[Definition~1.2 and formula~(1.8)]{Abe-Dimca-Pokora-hierarchy},
where $I_\nu$ is the Bourbaki ideal associated with a minimal syzygy
$\nu\in\syz(J_\A)_e$. This number is independent of the choice of
$\nu$. Our first main result relates the Bourbaki degree to the reduced
characteristic polynomial and gives sharp bounds in terms of the
initial degree and the intersection multiplicities.

\begin{mthm}
\label{main:characteristic-bounds}
Let $\A$ be an essential arrangement of $n+1$ lines and set
$e=e_\A$. Then
\begin{equation}\label{formula1}
  \Bour(\A)=\overline{\chi}_\A(e)  
\end{equation}
and
\[
\Bour(\A)\leq\binom e2.
\]
The latter bound is sharp for every $e\geq1$.

If $\A$ is not free and $m_1(\A)\geq m_2(\A)$ are the two largest
intersection multiplicities, then
\[
\Bour(\A)
\leq
\binom e2-\binom{m_2(\A)-1}{2},
\qquad
t(\A)\leq e-m_2(\A)+1.
\]
If $e\neq n+1-m_1(\A)$, then
\[
\Bour(\A)
\leq
\binom e2-\binom{m_1(\A)-1}{2},
\qquad
t(\A)\leq e-m_1(\A)+1.
\]
Moreover, in the nonfree case,
\[
\Bour(\A)=\binom e2
\]
if and only if $t(\A)=e-1$, and this is also equivalent to
$m_2(\A)=2$. In this case every Bourbaki ideal associated with a
minimal syzygy has the resolution
\[
0\longrightarrow R(-e)^{e-1}
\longrightarrow R(-(e-1))^e
\longrightarrow I_\nu
\longrightarrow0.
\]
\end{mthm}

The identity~\eqref{formula1} is one of the main points of the paper. The polynomial
$\overline{\chi}_\A(t)$ is determined by the intersection lattice,
whereas $e_\A$ need not be. Thus, once the lattice is fixed, the
possible variation of the Bourbaki degree is governed by the initial
degree of the first syzygy module. This simple formula will also be
useful later when we study the combinatorial behavior of the Bourbaki
degree.

The upper bound in Main Theorem~\ref{main:characteristic-bounds} is
substantially stronger than the general bound known for reduced plane
curves. In \cite[Theorem~2.10]{JNSBourbaki} it was proved that
\[
\Bour(X)\leq e^2
\]
for an arbitrary reduced plane curve $X$. For line arrangements we
replace the square by the triangular number
\[
\binom e2.
\]
Moreover, this is the best possible bound depending only on $e$, since
generic arrangements attain equality for every possible value of
$e$. Thus any further improvement must use additional information
about the arrangement.

The intersection multiplicities provide precisely such information.
The second largest multiplicity forces the Bourbaki degree to drop
from its maximal possible value by at least
\[
\binom{m_2(\A)-1}{2}.
\]
Under the additional assumption $e\neq n+1-m_1(\A)$, the maximal
multiplicity gives the stronger correction
\[
\binom{m_1(\A)-1}{2}.
\]
At the same time, these multiplicities control the type of the
arrangement. For instance, if $m_2(\A)=e$, then $t(\A)\leq1$, and
hence $t(\A)=1$; consequently, $\A$ is plus-one generated. At the
opposite extreme, the maximal value $\Bour(\A)=\binom e2$ occurs
exactly when $m_2(\A)=2$, or equivalently, when there is at most one
intersection point of multiplicity greater than two. In this case
$t(\A)=e-1$, and the associated Bourbaki ideals have a linear
Hilbert--Burch resolution. Thus the equality case of the Bourbaki
bound has both a geometric and a homological description.

As another consequence, Theorem~\ref{thm:low-degree-bourbaki-classification}
gives a complete description of the Bourbaki degree and the type for
arrangements with $e_\A\leq3$, using the known classifications of
arrangements with small initial degree.

The problem of finding optimal lower bounds for the global Tjurina
number of line arrangements remains open in general~\cite{Dimca-free-open-problems}. Related recent work includes
\cite{Dimca-remarks-freeness, Dimca-local-derivations,
Dimca-Kuhne-Pokora,Pokora-Tjurina-bounds}. In this direction, the above Main
Theorem gives new lower bounds for
$\tau(\A)$. Indeed,
\[
\tau(\A)=n(n-e)+e^2-\Bour(\A),
\]
and therefore the sharp triangular bound yields
\[
\tau(\A)\geq
n(n-e)+\binom{e+1}{2}.
\]
The correction term $\binom{e+1}{2}$ is optimal among those depending
only on $e$. For a nonfree arrangement, the multiplicity refinement
gives
\[
\tau(\A)\geq
n(n-e)+\binom{e+1}{2}
+\binom{m_2(\A)-1}{2},
\]
and, if $e\neq n+1-m_1(\A)$,
\[
\tau(\A)\geq
n(n-e)+\binom{e+1}{2}
+\binom{m_1(\A)-1}{2}.
\]

These inequalities sharpen the lower bounds obtained by Dimca in
\cite[Theorem~1.3]{Dimca-minimal-Tjurina}. In our notation, the first
bound in that result is
\[
\tau(\A)\geq
n(n-e)+\binom e2+\binom{m_2(\A)}2+1.
\]
The difference between our correction term and Dimca's is exactly
\[
e-m_2(\A)\geq0.
\]
Thus our estimate is never weaker and is strictly stronger whenever
$m_2(\A)<e$. Likewise, under the assumption
$e\neq n+1-m_1(\A)$, the improvement over the corresponding bound
involving the maximal multiplicity is exactly
\[
e-m_1(\A)\geq0.
\]
The refined bound is attained by a natural family of arrangements
with exactly two multiple points, as shown in
Proposition~\ref{prop:two-multiple-points}.

For free line arrangements, Pokora recently obtained a lower bound
for the global Tjurina number depending only on the number of lines
\cite[Corollary~4.2]{Pokora-Tjurina-bounds}. This result is of a
different nature, since it assumes freeness and does not involve the
initial degree or the intersection multiplicities. Its relation with
the Bourbaki formula and with the bounds above is discussed in
Remark~\ref{rem:comparison-dimca-tjurina}. 

\medskip

Over $\CC$, the Bourbaki degree is also closely related to the freeness defect,
another numerical measure of nonfreeness, defined by
\[
\nu(\A)
=
\max_q\dim_k
\bigl(J_\A^{\mathrm{sat}}/J_\A\bigr)_q.
\]
Using the formulas of Dimca
\cite[Theorem~1.7]{Dimca-minimal-Tjurina}, we prove in
Proposition~\ref{prop:bourbaki-freeness-defect} the precise relation
\[
\nu(\A)
=
\begin{cases}
\Bour(\A),
&\text{if }e\leq\left\lfloor\frac n2\right\rfloor,\\[2mm]
\Bour(\A)
-\left(
\left\lfloor\frac{n^2}{4}\right\rfloor-e(n-e)
\right),
&\text{if }e>\left\lfloor\frac n2\right\rfloor.
\end{cases}
\]
Thus, in the first range, the Bourbaki degree agrees exactly with the
freeness defect. In the second range, the difference between the two
invariants depends only on $n$ and $e$. In particular, the bounds
obtained above for the Bourbaki degree give corresponding bounds for
$\nu(\A)$.

The refined bounds involving the intersection multiplicities also
lead to a new sufficient condition for Terao's conjecture.

\begin{mthm}
\label{main:terao}
Assume that $k=\CC$. Let $\A$ be an essential free arrangement of $n+1$ lines with
exponents $(e,n-e)$, where $e\leq n-e$, and let $\B$ be an
arrangement such that $L(\B)\simeq L(\A)$. If
\[
\binom{e+1}{2}
<
n+\binom{m_2(\A)-1}{2},
\]
then $\B$ is free, necessarily with exponents $(e,n-e)$.

If, in addition, $e\neq n+1-m_1(\A)$, the same conclusion holds under
\[
\binom{e+1}{2}
<
n+\binom{m_1(\A)-1}{2}.
\]
In particular, Terao's conjecture holds for $\A$ whenever
\[
\binom{e+1}{2}<n+1.
\]
\end{mthm}

Main Theorem~\ref{main:terao} gives a criterion for Terao's conjecture
which takes the intersection multiplicities into account. The first
condition involves the second largest multiplicity $m_2(\A)$, while,
under the additional assumption $e\neq n+1-m_1(\A)$, the maximal
multiplicity gives a stronger criterion. Since these multiplicities
are determined by the intersection lattice, the result uses
combinatorial information that is not present in criteria depending
only on the number of lines and the smaller exponent. In particular,
it may apply even when such uniform numerical criteria do not.

As a uniform consequence, Terao's conjecture holds whenever
\[
\binom{e+1}{2}<n+1,
\]
or equivalently,
\[
n\geq\binom{e+1}{2}.
\]
For comparison, Dimca's criterion
\cite[Corollary~3.1]{Dimca-minimal-Tjurina} is
\[
e<\frac{\sqrt{8n+49}-3}{2},
\]
which, since $e$ is an integer, is equivalent to
\[
n\geq\binom{e+1}{2}+e-4.
\]
Thus the two criteria agree for $e=4$, while for every $e\geq5$
our criterion lowers the required value of $n$ by exactly $e-4$.
The first strict improvement occurs for $e=5$ and $n=15$, that is,
for arrangements of sixteen lines. More importantly, the criterion
involving the intersection multiplicities can apply even when neither
uniform numerical condition is satisfied.

\medskip

The sharp bound
\[
\Bour(\A)\leq\binom{e_\A}{2}
\]
also leads naturally to the study of addition and deletion. We measure
the gap from equality by
\[
\Delta(\A):=
\binom{e_\A}{2}-\Bour(\A).
\]
Thus $\Delta(\A)\geq0$ by Main
Theorem~\ref{main:characteristic-bounds}, and we call $\A$
\emph{extremal} when $\Delta(\A)=0$.

Let $\A'$ be an essential arrangement of $n$ lines and let
\[
\A=\A'\cup\{H\}.
\]
If $s$ is the number of distinct points in which the lines of $\A'$
meet $H$, set
\[
c_H:=n-s.
\]
Geometrically, $c_H$ measures the collisions of the intersections on
the added line. More precisely, if a point $p\in H$ has multiplicity
$m_p$ in $\A'$, then it contributes $m_p-1$ to $c_H$.

Our addition--deletion formula determines exactly how the defect
changes when a line is added.

\begin{mthm}
\label{main:addition-deletion}
With the notation above, set $e'=e_{\A'}$ and $e=e_\A$. Then
$e=e'$ or $e=e'+1$, and
\[
\Delta(\A)-\Delta(\A')
=
\begin{cases}
c_H-(n-e'-1),&\text{if }e=e',\\[2mm]
c_H,&\text{if }e=e'+1.
\end{cases}
\]
In particular,
\[
\Delta(\A)\geq\Delta(\A').
\]
Hence extremality is inherited by every essential deletion.

Moreover, an essential arrangement is extremal if and only if it can
be obtained from a triangle by successive additions of the following
two types:
\begin{enumerate}
\item a line containing no intersection point of the arrangement
constructed so far;
\item a line through a point of maximal multiplicity and through no
other intersection point of the arrangement constructed so far.
\end{enumerate}
\end{mthm}
The monotonicity of $\Delta$ is one of the main consequences of this
formula. It shows that the defect can only increase when lines are
added, and therefore every essential deletion of an extremal
arrangement is again extremal. The equality cases also have a simple
geometric interpretation. If $e=e'+1$, equality
$\Delta(\A)=\Delta(\A')$ holds precisely when $c_H=0$, that is, when
the added line contains no intersection point of $\A'$. In the
 case $e=e'$, equality holds precisely when
$c_H=n-e'-1$. When $\A'$ is extremal, the latter condition has a
geometric characterization: the added line passes through a point of
maximal multiplicity and through no other intersection point of
$\A'$. Thus, within the extremal class, the two operations appearing
in Main Theorem~\ref{main:addition-deletion} are exactly the additions
that preserve the defect. Starting from a triangle, this gives a
constructive characterization of all extremal arrangements.

\medskip

The last part of the paper concerns the combinatorial behavior of the
Bourbaki degree. Since the reduced characteristic polynomial is
determined by the intersection lattice, the formula~\eqref{formula1}
shows that, within a fixed lattice class, any variation of the
Bourbaki degree must come from the initial degree $e_\A$. This leads
to an exact comparison for two realizations of the same lattice.

\begin{mthm}
\label{main:combinatorial}
Let $\A$ and $\B$ be essential arrangements of $n+1$ lines with
$L(\A)\simeq L(\B)$. Then
\[
\Bour(\A)-\Bour(\B)
=
(e_\A-e_\B)(e_\A+e_\B-n).
\]
Consequently, $\Bour(\A)=\Bour(\B)$ if and only if
$e_\A=e_\B$ or $e_\A+e_\B=n$.

Over $\CC$, the Bourbaki degree is determined by the intersection
lattice for arrangements with at most eight lines. For every number
of lines at least nine, there exist arrangements with isomorphic
intersection lattices and different Bourbaki degrees. Hence nine is
the smallest number of lines for which the Bourbaki degree can fail
to be combinatorially determined.
\end{mthm}

The statement on the number of lines is sharp. The nine-line examples come from the recent Ziegler pair of Dimca and
Pokora \cite{Dimca-Pokora-nine}, for which the two arrangements have
the same intersection lattice but different initial degrees. The
formula above then shows that their Bourbaki degrees are different. 
By successive generic additions, using
Corollary~\ref{cor:generic-line-bourbaki-addition}, one obtains
examples for every larger number of lines. Thus the first possible failure of the combinatorial
invariance of the Bourbaki degree occurs with nine lines.

The second possibility allowed by Main
Theorem~\ref{main:combinatorial} is more subtle. Suppose that two
arrangements with the same intersection lattice have different
initial degrees but the same Bourbaki degree. Set
\[
q=\min\{e_\A,e_\B\},
\qquad
b=\Bour(\A)=\Bour(\B).
\]
Then necessarily
\[
\{e_\A,e_\B\}=\{q,n-q\}.
\]
Moreover, both arrangements are nonfree, and, writing $m_1$ for their
common maximal intersection multiplicity, one has
\[
\binom{n-2q+1}{2}
\leq
b
\leq
\binom q2-\binom{m_1-1}{2}.
\]
In particular,
\[
\frac{n+1}{3}<q<\frac n2.
\]
These restrictions show that this exceptional case cannot occur with
fewer than ten lines. For ten lines the only possible initial degrees
are $4$ and $5$. In this case the maximal intersection multiplicity
must be $4$, and the common Bourbaki degree is either $1$ or $2$.
Thus, for ten lines, the problem is reduced to a small number of
possible weak-combinatorial types. Whether this exceptional situation
actually occurs is left as an open question.

The characteristic-polynomial formula also gives a direct connection
with modularity. Let $p$ be an intersection point of multiplicity
$m$, and define
\[
\delta_p(\A)
:=
\sum_{\substack{r\neq p\\ \overline{pr}\notin\A}}
(m_r-1).
\]
Then $\delta_p(\A)=0$ precisely when $p$ is modular. We prove the
identity
\[
\overline{\chi}_\A(t)
=
(t-m+1)\bigl(t-(n+1-m)\bigr)+\delta_p(\A).
\]
Consequently, if the syzygy associated with $p$ has minimal degree,
that is, if
\[
e_\A=n+1-m,
\]
then
\[
\Bour(\A)=\delta_p(\A).
\]
Thus, in this case, the Bourbaki degree measures exactly the failure
of the point $p$ to be modular.

\section*{Acknowledgments}
This work grew out of the second author's doctoral dissertation at IASBS. The authors are grateful to  Rashid Zaare-Nahandi, Felipe Monteiro and  Tomás S.~R.~Silva for insightful discussions and helpful suggestions. The first author was partially supported by FAPESP Grant No.~2026/08480-1

\section{Preliminaries}
\label{sec:preliminaries}
We begin by recalling the basic facts on Bourbaki ideals of plane curves and the combinatorial invariants of line arrangements that will be used throughout the paper.

\subsection{Bourbaki degree}
Let $R=k[x,y,z]$ be the standard graded polynomial ring over an algebraically closed field $k$ of characteristic zero, with irrelevant maximal ideal $\fm=(x,y,z)$. Let $f\in R$ be a reduced homogeneous polynomial of degree $d+1\geq2$, defining a plane curve $X=V(f)\subseteq\PP^2$, and let $J_f=(f_x,f_y,f_z)$ be its gradient ideal. Since $f$ is reduced, $\hht J_f\geq2$. If $X$ is singular, then $\hht J_f=2$ and $\proj(R/J_f)$ is the singular subscheme of $X$; if $X$ is smooth, then $J_f$ is $\fm$-primary.

The syzygy module of $J_f$ is defined by the graded exact sequence
\[
0\longrightarrow\syz(J_f)\longrightarrow R^3
\xrightarrow{(f_x,f_y,f_z)}J_f(d)\longrightarrow0.
\]
Thus $\syz(J_f)=\bigoplus_{q\geq0}\syz(J_f)_q$ is a reflexive graded $R$-module of rank two. We set
\[
e=\indeg\bigl(\syz(J_f)\bigr)
=\min\{q:\syz(J_f)_q\neq0\}
\]
and assume throughout that $e\geq1$, thereby excluding the cone case.

If $X$ is singular, then $\dim(R/J_f)=1$ and by~\cite[Proposition~3.2]{JNSBourbaki} $$\deg(R/J_f)=\tau(X),$$ where $\tau(X)=\sum_{p\in\Sing(X)}\tau_p(X)$ is the global Tjurina number of $X$.

Let $\Der_k(R)$ denote the $R$-module of $k$-linear derivations of $R$, and set
\[
\Der_f(R)=\{\delta\in\Der_k(R):\delta(f)\in(f)\},
\qquad
\Der_f(R)_0=\{\delta\in\Der_k(R):\delta(f)=0\}.
\]
Following Saito \cite{SaitoLogarithmic}, the curve $X$ is called \emph{free} if $\Der_f(R)$ is a free $R$-module. Under the natural identification $\Der_k(R)\simeq R\partial_x\oplus R\partial_y\oplus R\partial_z$, one has $\Der_f(R)_0\simeq\syz(J_f).$ Since $f$ is homogeneous and $\operatorname{char}(k)=0$, Euler's identity yields $\Der_f(R)=\Der_f(R)_0\oplus R\mathcal{E},$ where $\mathcal E=x\partial_x+y\partial_y+z\partial_z$ is the Euler derivation. Consequently, $X$ is free if and only if $\syz(J_f)$ is a free $R$-module of rank two. In this case,
\[
\syz(J_f)\simeq R(-a_1)\oplus R(-a_2),
\qquad a_1+a_2=d,
\]
where $(a_1,a_2)$ are the exponents of $X$. We always order them so that $a_1\leq a_2$; hence $a_1=e$ and $a_2=d-e$.

We now recall the Bourbaki construction for reduced plane curves from \cite[Section~2]{JNSBourbaki}, in the form needed below.

Choose a minimal homogeneous generator $\nu\in\syz(J_f)_e$. The inclusion $R(-e)\xrightarrow{\nu}\syz(J_f)$ has torsionfree cokernel of rank one. Hence there is a homogeneous ideal $I_\nu\subseteq R$ and an exact sequence
\begin{equation}
\label{eq:bourbaki-sequence}
0\longrightarrow R(-e)
\xrightarrow{\ \nu\ }
\syz(J_f)
\xrightarrow{\ \pi_\nu\ }
I_\nu(e-d)
\longrightarrow0.
\end{equation}
We call \eqref{eq:bourbaki-sequence} the \emph{Bourbaki sequence} associated with $\nu$, and $\pi_\nu$ the corresponding \emph{Bourbaki map}. Thus $\ker(\pi_\nu)=R\nu$.

If $X$ is free, then $I_\nu=R$. Otherwise, $I_\nu$ is a proper homogeneous perfect ideal of height two; we call it the \emph{Bourbaki ideal} associated with $\nu$, and
\[
B_\nu=\proj(R/I_\nu)\subseteq\PP^2
\]
the corresponding \emph{Bourbaki scheme}.

The Bourbaki ideal can also be read directly from a minimal resolution of $\syz(J_f)$. Suppose that
\[
0\longrightarrow F_1
\xrightarrow{\left[\begin{smallmatrix}\lambda\\ \psi\end{smallmatrix}\right]} R(-e)\oplus F_0 \xrightarrow{(\nu,\varphi)}
\syz(J_f)
\longrightarrow0
\]
is a minimal graded free resolution obtained from a minimal generating set containing $\nu$. Then $I_\nu$ has the minimal resolution
\begin{equation}\label{HBB}
0\longrightarrow F_1(d-e)
\xrightarrow{\ \psi\ } F_0(d-e) \longrightarrow I_\nu \longrightarrow0.
\end{equation}
Thus, in the nonfree case, $\psi$ is a Hilbert--Burch matrix for $I_\nu$.

The \emph{Bourbaki degree} of $X$ is
\[
\Bour(X)=
\begin{cases}
\deg(R/I_\nu),&\text{if $X$ is not free},\\ 0,&\text{if $X$ is free}.
\end{cases}
\]
Although $I_\nu$ may depend on the choice of $\nu$, its degree does not. More precisely,
\begin{equation}
\label{eq:bourbaki-degree-formula}
\Bour(X)=d^2+e(e-d)-\tau(X).
\end{equation}
For a smooth curve we use the convention $\tau(X)=0$. If $X$ is singular, then $\Bour(X)=0$ if and only if $X$ is free. We also recall the notions of nearly free and plus-one generated curves. A reduced curve of degree $d+1$ is \emph{plus-one generated} if $\syz(J_f)$ has three minimal homogeneous generators of degrees $d_1\leq d_2\leq d_3$ satisfying $d_1+d_2=d+1$. The integers $(d_1,d_2,d_3)$ are its exponents, and $d_3$ is its \emph{level}. A plus-one generated curve is nearly free precisely when $d_2=d_3$.

If $X$ is plus-one generated with exponents $(d_1,d_2,d_3)$, then $\Bour(X)=d_3-d_2+1$. In particular, $\Bour(X)=1$ if and only if $X$ is nearly free, while $\Bour(X)=2$ if and only if $X$ is plus-one generated with exponents $(e,d-e+1,d-e+2)$, where $e\leq(d+1)/2$.


\subsection{Line arrangements}
\label{subsec:line-arrangements}

A \emph{line arrangement} in $\PP^2$ is a finite collection $\A=\{L_1,\ldots,L_{n+1}\}$ of distinct projective lines. Choose pairwise nonproportional linear forms $\ell_i\in R_1$ such that $L_i=V(\ell_i)$, and set
\[
f_\A=\prod_{i=1}^{n+1}\ell_i,\qquad X_\A=V(f_\A),\qquad J_\A=J_{f_\A}.
\]
We write $e_\A:=\indeg\bigl(\syz(J_\A)\bigr)$. The arrangement $\A$ is \emph{essential} if its lines are not all concurrent.  If all the lines pass through one point, then $\A$ is a \emph{pencil}; after a linear change of coordinates, $f_\A\in k[x,y]$, and hence $e_\A=0$. Conversely, if $e_\A=0$, then a nonzero constant directional derivative of $f_\A$ vanishes. After a linear change of coordinates, $f_\A$ is therefore independent of one variable, and $\A$ is a pencil. Thus $e_\A\geq1$ if and only if $\A$ is essential.

We use the standard lattice-theoretic terminology and conventions of \cite[Chapter~2]{Orlik-Terao}. Let $c\A$ denote the central arrangement in $k^3$ defined by $\ell_1,\ldots,\ell_{n+1}$. Its \emph{intersection lattice} is
\[
L(\A):=L(c\A)
=
\left\{
\bigcap_{i\in I}V(\ell_i): I\subseteq\{1,\ldots,n+1\} \right\},
\]
ordered by reverse inclusion, with the empty intersection understood as $k^3$. The elements of $L(\A)$ are called \emph{flats}; the rank of a flat $X\in L(\A)$ is $\operatorname{rk}(X)=\operatorname{codim}_{k^3}(X)$. Two arrangements $\A$ and $\B$ are \emph{combinatorially equivalent} if $L(\A)\simeq L(\B)$ as ranked lattices. Notice that $\A$ is essential precisely when $c\A$ has rank three.

The rank-two elements of $L(\A)$ are the one-dimensional flats of $c\A$ and, after projectivization, correspond to the intersection points of $\A$. We denote their set by $L_2(\A)$. For $p\in L_2(\A)$, let $m_p=m_p(\A):=|\{L\in\A:p\in L\}|$ be the multiplicity of $\A$ at $p$. A point with $m_p=j$ is called a $j$-fold point. We set
\[
n_j=n_j(\A):=
|\{p\in L_2(\A):m_p=j\}|, \qquad m(\A):=\max_{p\in L_2(\A)}m_p.
\]
The sequence $(n_2,n_3,\ldots)$ is the \emph{weak combinatorics} of $\A$. It is determined by $L(\A)$, although in general it does not determine the full intersection lattice. The arrangement is \emph{generic} if every intersection point is a double point.

Let $p\in L_2(\A)$ be a point of multiplicity $m_p$. In suitable local coordinates $u,v$ centered at $p$, a reduced local equation of $X_\A$ is, up to multiplication by a unit,
\[
g_p(u,v)
=
\lambda_1(u,v)\cdots\lambda_{m_p}(u,v),
\]
where the $\lambda_i$ are pairwise nonproportional linear forms. Thus $(X_\A,p)$ is an ordinary $m_p$-fold point with $m_p$ smooth branches having distinct tangent directions.

The local Milnor and Tjurina numbers are
\[
\mu_p(\A)
=
\dim_k
\frac{k[[u,v]]}{((g_p)_u,(g_p)_v)},
\qquad
\tau_p(\A)
=
\dim_k
\frac{k[[u,v]]}{(g_p,(g_p)_u,(g_p)_v)}.
\]
Since $g_p$ is reduced, $(g_p)_u$ and $(g_p)_v$ have no common nonconstant factor. They therefore form a homogeneous regular sequence of degree $m_p-1$ in $k[u,v]$. The corresponding quotient is Artinian, so completion preserves its length. Hence
\[
\mu_p(\A)
=
\dim_k
\frac{k[u,v]}{((g_p)_u,(g_p)_v)}
=
(m_p-1)^2.
\]
On the other hand, Euler's identity gives $m_pg_p=u(g_p)_u+v(g_p)_v,$ so $g_p\in((g_p)_u,(g_p)_v)$. It follows that $\mu_p(\A)=\tau_p(\A)=(m_p-1)^2.$ In particular, every singularity of a line arrangement is quasi-homogeneous; see \cite{Saito} for the terminology. Therefore
\[
\mu(\A)=\tau(\A)
=
\sum_{p\in L_2(\A)}(m_p-1)^2
=
\sum_{j\geq2}n_j(j-1)^2.
\]
Thus the global Milnor and Tjurina numbers are determined by the weak combinatorics.

We shall also use the integer
\[
b_2(\A)
:=
\sum_{p\in L_2(\A)}(m_p-1)
=
\sum_{j\geq2}n_j(j-1).
\]
Since every pair of lines meets at exactly one point,
\[
\sum_{p\in L_2(\A)}\binom{m_p}{2}
=
\binom{n+1}{2}.
\]
Together with $(m_p-1)^2=2\binom{m_p}{2}-(m_p-1)$, this gives
\begin{equation}
\label{eq:tjurina-b2-line-arrangement}
\tau(\A)=n(n+1)-b_2(\A).
\end{equation}
Thus both $b_2(\A)$ and $\tau(\A)$ are determined by the weak combinatorics. For a generic arrangement, $b_2(\A)=\tau(\A)=\binom{n+1}{2}$, whereas for a pencil, $b_2(\A)=n$ and $\tau(\A)=n^2$.

The construction of a syzygy in $\syz(J_\A)$ associated with an intersection point is classical; see \cite[Section~2.2]{Dimca-pencils}, \cite[Lemma~2.2]{Dimca-minimal-Tjurina}, and \cite[Theorem~1.3]{Dimca-local-derivations}. We record the form needed below and show that syzygies associated with distinct intersection points cannot be proportional.

\begin{Proposition}
\label{prop:syzygies-associated-with-points}
Let $\A$ be an arrangement of $n+1$ lines and let $p\in L_2(\A)$ have multiplicity $m_p$. Then there exists a syzygy $\rho_p\in\syz(J_\A)_{n+1-m_p}$ whose coordinates have no nonconstant common factor. Consequently,
\begin{equation}
\label{eq:initial-degree-multiplicity-bound}
e_\A\leq n+1-m_p
\end{equation}
for every $p\in L_2(\A)$, and hence
\begin{equation}
\label{eq:initial-degree-maximal-multiplicity}
e_\A\leq n+1-m(\A).
\end{equation}
Moreover, if $p,q\in L_2(\A)$ are distinct, then $\rho_p$ and $\rho_q$ are not nonzero scalar multiples.
\end{Proposition}
\begin{proof}
After a linear change of coordinates, we may assume that $p=(1:0:0)$ and write
\[
f_\A=g_p(y,z)h_p(x,y,z), \qquad \deg(g_p)=m_p,\qquad \deg(h_p)=n+1-m_p,
\]
where $g_p$ is the product of the equations of the lines through $p$. Euler's identities give
\[
yg_{p,y}+zg_{p,z}=m_pg_p, \qquad xh_{p,x}+yh_{p,y}+zh_{p,z}=(n+1-m_p)h_p.
\]
A direct computation then shows that
\[
\rho_p
:= (xh_{p,x}-(n+1)h_p,\;yh_{p,x},\;zh_{p,x}) \in\syz(J_\A)_{n+1-m_p}.
\]
This is the syzygy associated with $p$ appearing in \cite[Theorem~1.2, formula~(2.4)]{Dimca-pencils}.

We first show that the coordinates of $\rho_p$ have no nonconstant common factor. Suppose that an irreducible form $D$ divides its three coordinates. Since $D$ divides both $yh_{p,x}$ and $zh_{p,x}$ and $\gcd(y,z)=1$, one has $D\mid h_{p,x}$. The first coordinate then gives $D\mid h_p$. If $h_p=1$, this is already a contradiction. Otherwise, $D=\ell$ is a linear factor of $h_p$. Since the line $V(\ell)$ does not pass through $p$, one has $\ell_x\neq0$, while
\[
h_{p,x} \equiv \ell_x\frac{h_p}{\ell} \pmod{\ell}.
\]
As $h_p$ is reduced, $\ell\nmid h_p/\ell$, and hence $\ell\nmid h_{p,x}$, a contradiction. Therefore the coordinates of $\rho_p$ have no nonconstant common factor, and \eqref{eq:initial-degree-multiplicity-bound} and \eqref{eq:initial-degree-maximal-multiplicity} follow.

It remains to compare the syzygies associated with distinct points. Write $p=(a:b:c)$ and let
\[
\partial_p
=
a\partial_x+b\partial_y+c\partial_z.
\]
If $h_p$ denotes the product of the equations of the lines not passing through $p$, the same construction in arbitrary coordinates gives $\rho_p=(A_p,B_p,C_p),$ where
\[
\begin{aligned}
xB_p-yA_p&=(n+1)h_p(ay-bx),\\
xC_p-zA_p&=(n+1)h_p(az-cx),\\
yC_p-zB_p&=(n+1)h_p(bz-cy).
\end{aligned}
\]
Consider the ideal of $2\times2$ minors
\[
I_2
\begin{pmatrix}
x&y&z\\ A_p&B_p&C_p
\end{pmatrix}.
\]
Since $I(p)=(ay-bx,\;az-cx,\;bz-cy)$ is the homogeneous ideal of $p$, we obtain
\[
I_2
\begin{pmatrix}
x&y&z\\ A_p&B_p&C_p
\end{pmatrix}
=
h_pI(p).
\]

Now suppose that $p,q\in L_2(\A)$ and that $\rho_p=\lambda\rho_q$ for some $\lambda\in k^*$. The corresponding ideals of $2\times2$ minors are equal, and hence $h_pI(p)=h_qI(q).$ Since $I(p)$ and $I(q)$ are height-two linear prime ideals, the greatest common divisor of the generators of $h_pI(p)$ is $h_p$, up to a nonzero scalar, and similarly for $h_qI(q)$. It follows that $h_p$ and $h_q$ are proportional. Cancelling this common factor from the preceding equality gives $I(p)=I(q),$ and therefore $p=q$. Thus syzygies associated with distinct intersection points cannot be nonzero scalar multiples.
\end{proof}

Let $\mu_L:L(\A)\to\ZZ$ be the M\"obius function, characterized by
\[
\mu_L(k^3)=1,
\qquad
\sum_{Y\leq X}\mu_L(Y)=0
\quad\text{for }X\neq k^3.
\]
The \emph{characteristic polynomial} of $\A$ is
\[
\chi_\A(t)
:=
\sum_{X\in L(\A)}\mu_L(X)t^{\dim X}.
\]
Each hyperplane $H_i\in c\A$ has M\"obius value $\mu_L(H_i)=-1$. Let $p\in L_2(\A)$ have multiplicity $m_p$, and let $X_p\subseteq k^3$ be the corresponding one-dimensional flat. Since exactly $m_p$ hyperplanes contain $X_p$, the defining recursion gives
\[
\mu_L(X_p)
=
-\left(
\mu_L(k^3)+\sum_{H_i\supset X_p}\mu_L(H_i)
\right)
=
m_p-1.
\]
Hence the contributions of the flats of dimensions three, two, and one to $\chi_\A(t)$ are, respectively, $t^3,\, -(n+1)t^2,\, b_2(\A)t$. Set
\[
b_3(\A)
:= -\sum_{\substack{X\in L(\A)\\ \operatorname{rk}(X)=3}}
\mu_L(X).
\]
If $\A$ is essential, the unique rank-three flat is $\{0\}$, whereas for a pencil there is no rank-three flat. Thus
\[
\chi_\A(t)
=
t^3-(n+1)t^2+b_2(\A)t-b_3(\A).
\]
Since $c\A$ is central and nonempty, $\chi_\A(1)=0$, and therefore $b_3(\A)=b_2(\A)-n.$ Consequently,
\[
\chi_\A(t)
=
t^3-(n+1)t^2+b_2(\A)t-\bigl(b_2(\A)-n\bigr).
\]

The \emph{reduced characteristic polynomial} is defined by
\[
\overline{\chi}_\A(t)
:=
\frac{\chi_\A(t)}{t-1}.
\]
Hence
\begin{equation}
\label{eq:reduced-characteristic-polynomial}
\overline{\chi}_\A(t)
=
t^2-nt+b_2(\A)-n.
\end{equation}
Both $\chi_\A(t)$ and $\overline{\chi}_\A(t)$ are determined by the intersection lattice.

Finally, if $\A$ is free with projective exponents $(a_1,a_2)$, Terao's factorization theorem \cite[Theorem~4.137]{Orlik-Terao} gives
\[
\chi_\A(t)
=
(t-1)(t-a_1)(t-a_2), \qquad
\overline{\chi}_\A(t)
=
(t-a_1)(t-a_2).
\]

\section{The Bourbaki degree of line arrangements}
\label{sec:bourbaki-degree-line-arrangements}

Let $\A=\{L_1,\ldots,L_{n+1}\}$ be an essential arrangement of $n+1\geq3$ distinct lines in $\PP^2$, with defining polynomial $f_\A=\ell_1\cdots\ell_{n+1}$. We write $J_\A=J_{f_\A}$ and $e=e_\A:=\indeg\bigl(\syz(J_\A)\bigr)$. Since $\A$ is essential, $e\geq1$. Choose a minimal homogeneous generator $\nu\in\syz(J_\A)_e$. Since $\deg(f_\A)=n+1$, the Bourbaki construction of the preceding section gives
\[
0\longrightarrow R(-e)\xrightarrow{\ \nu\ }\syz(J_\A)
\xrightarrow{\ \pi_\nu\ }I_\nu(e-n)\longrightarrow0.
\]
If $\A$ is free, then $I_\nu=R$ and $\Bour(\A)=0$. Otherwise, $I_\nu$ is a proper saturated perfect ideal of height two and $\Bour(\A)=\deg(R/I_\nu)$.

Our aim in this section is to study the Bourbaki degree in terms of the geometry and combinatorics of the arrangement.
\subsection{The characteristic-polynomial formula}
\label{subsec:characteristic-polynomial-formula}

For an arrangement of $n+1$ lines, the general formula for the Bourbaki degree becomes
\begin{equation}
\label{eq:bourbaki-tjurina-arrangement}
\Bour(\A)=n^2+e(e-n)-\tau(\A).
\end{equation}
Together with the formula for the global Tjurina number, this gives the following relation with the reduced characteristic polynomial.

\begin{Theorem}
\label{thm:bourbaki-characteristic-polynomial}
Let $\A$ be an essential arrangement of $n+1$ lines in $\PP^2$ and let $e=e_\A$. Then
\[
\Bour(\A)=\overline{\chi}_\A(e).
\]
\end{Theorem}

\begin{proof}
By \eqref{eq:bourbaki-tjurina-arrangement} and \eqref{eq:tjurina-b2-line-arrangement},
\[
\Bour(\A)
=n^2+e(e-n)-n(n+1)+b_2(\A)
=e^2-ne+b_2(\A)-n.
\]
On the other hand, $\overline{\chi}_\A(t)=t^2-nt+b_2(\A)-n.$ Evaluating at $t=e$ gives the result.
\end{proof}

Thus the Bourbaki degree is the value of a combinatorially determined polynomial at the generally noncombinatorial integer $e_\A$. In particular, $\A$ is free if and only if $\overline{\chi}_\A(e)=0$, and nearly free if and only if $\overline{\chi}_\A(e)=1$. Since $\overline{\chi}_\A(t)-\overline{\chi}_\A(e)
=(t-e)(t-n+e)$, these conditions are equivalent, respectively, to
\[
\overline{\chi}_\A(t)=(t-e)(t-n+e)
\quad\text{and}\quad
\overline{\chi}_\A(t)=(t-e)(t-n+e)+1.
\]

\begin{Example}
\label{ex:generic-arrangements-bourbaki}
Let $\A$ be a generic arrangement of $n+1$ lines, with $n\geq3$. Then every singular point of $\A$ is a double point, so $b_2(\A)=\binom{n+1}{2}$ and
\[
\overline{\chi}_\A(t)=t^2-nt+\binom n2.
\]
By~\cite[Example~2.2(i)]{DimcaSernesi}, one has $e_\A=n-1$. Hence
\[
\Bour(\A)
=\overline{\chi}_\A(n-1)
=\binom{n-1}{2}.
\]

By \cite{YuzvinskyGeneric}, taking into account the grading shift between logarithmic derivations and $\syz(J_\A)$, one has
\[
0\longrightarrow R(-n)^{n-2}
\longrightarrow R(-(n-1))^n \longrightarrow\syz(J_\A)\longrightarrow0.
\]
Thus, for every $0\neq\nu\in\syz(J_\A)_{n-1}$, by~\eqref{HBB} the associated Bourbaki ideal has the minimal resolution
\[
0\longrightarrow R(-(n-1))^{n-2}
\longrightarrow R(-(n-2))^{n-1} \longrightarrow I_\nu\longrightarrow0.
\]
In particular, every Bourbaki scheme associated with a minimal syzygy of $J_\A$ has degree $\binom{n-1}{2}$ and the same Hilbert function and graded Betti numbers.
\end{Example}

\subsection{Bounds for line arrangements}
\label{subsec:bounds-line-arrangements}

For a singular reduced plane curve $X$, the general bound $\Bour(X)\leq e^2$ was proved in \cite[Theorem~2.10]{JNSBourbaki}. For line arrangements, we obtain a sharper bound and then refine it using the largest intersection multiplicities.

For each $p\in L_2(\A)$, let $\rho_p\in\syz(J_\A)_{n+1-m_p}$ be the syzygy associated with $p$ in Proposition~\ref{prop:syzygies-associated-with-points}. Recall that the syzygies associated with distinct intersection points are not nonzero scalar multiples.

We first prove the regularity estimate needed for these bounds. We recall the basic facts on Castelnuovo--Mumford regularity that will be used below. If a finitely generated graded $R$-module $M$ has a minimal graded free resolution
\[
\cdots\longrightarrow\bigoplus_jR(-b_{ij}) \longrightarrow\cdots\longrightarrow \bigoplus_jR(-b_{0j})\longrightarrow M\longrightarrow0,
\]
then $\reg(M)=\max_{i,j}\{b_{ij}-i\}.$ In particular, $\reg(M(a))=\reg(M)-a$. Moreover, for an exact sequence $0\to M'\to M\to M''\to0$ one has $\reg(M'')\leq\max\{\reg(M),\reg(M')-1\}.$ If $I\subseteq R$ is a nonzero proper homogeneous ideal, then $\reg(R/I)=\reg(I)-1$.

\begin{Lemma}
\label{lem:regularity-bourbaki-ideal}
Assume that $\A$ is not free. Let $0\neq\nu\in\syz(J_\A)_e$ and let $I_\nu$ be the associated Bourbaki ideal. Then
\[
\reg(I_\nu)\leq e-1.
\]
Consequently,
\begin{equation}
\label{eq:bourbaki-postulation-degrees}
\begin{aligned}
\Bour(\A)
&=\HF_{R/I_\nu}(e-2)
=\binom e2-\dim_k(I_\nu)_{e-2}\\
&=\HF_{R/I_\nu}(e-1)
=\binom{e+1}{2}-\dim_k(I_\nu)_{e-1}.
\end{aligned}
\end{equation}
\end{Lemma}

\begin{proof}
In the notation of \cite{Schenck}, $D_0\simeq\syz(J_\A)$. By \cite[Corollary~3.5 and the discussion following it]{Schenck}, $\syz(J_\A)$ has a minimal graded free resolution
\[
0\longrightarrow\bigoplus_{j=1}^{q-2}R(-\beta_j)
\longrightarrow\bigoplus_{i=1}^{q}R(-\alpha_i) \longrightarrow\syz(J_\A)\longrightarrow0,
\]
where $\alpha_i\leq n-1$ and $\beta_j\leq n$. Hence $\reg(\syz(J_\A))\leq n-1$.

The Bourbaki sequence is
\[
0\longrightarrow R(-e)\longrightarrow\syz(J_\A)
\longrightarrow I_\nu(e-n)\longrightarrow0.
\]
Since every intersection point has multiplicity at least two, \eqref{eq:initial-degree-multiplicity-bound} gives $e\leq n-1$. Therefore
\[
\reg\bigl(I_\nu(e-n)\bigr)
\leq\max\{\reg(\syz(J_\A)),e-1\}\leq n-1.
\]
Since $\reg(I_\nu(e-n))=\reg(I_\nu)-e+n$, it follows that $\reg(I_\nu)\leq e-1$.

Since $I_\nu$ is a proper perfect ideal of height two, $R/I_\nu$ is a one-dimensional Cohen--Macaulay ring and $\reg(R/I_\nu)=\reg(I_\nu)-1\leq e-2$. Thus its Hilbert function agrees with its constant Hilbert polynomial, equal to $\deg(R/I_\nu)=\Bour(\A)$, in degrees $e-2$ and $e-1$. Hence
\[
\Bour(\A)
=\dim_kR_{e-2}-\dim_k(I_\nu)_{e-2}
=\dim_kR_{e-1}-\dim_k(I_\nu)_{e-1},
\]
which gives \eqref{eq:bourbaki-postulation-degrees}.
\end{proof}

For a nonfree arrangement, we also use its \emph{type} $t(\A):=\indeg(I_\nu).$ This invariant was introduced in \cite[Definition~1.2 and formula~(1.8)]{Abe-Dimca-Pokora-hierarchy}. It is independent of the choice of $\nu$ and records the first nonzero degree of the Bourbaki ideal. In particular, $t(\A)=1$ characterizes plus-one generated arrangements. We will use the type to relate the structure of $I_\nu$ to the Bourbaki degree and the intersection multiplicities of $\A$.

\begin{Theorem}
\label{thm:sharp-triangular-bourbaki-bound}
Let $\A$ be an essential arrangement and set $e=e_\A$. Then
\begin{equation}
\label{eq:sharp-triangular-bourbaki-bound}
\Bour(\A)\leq\binom e2.
\end{equation}
The bound is sharp for every $e\geq1$. If $\A$ is not free, then $1\leq t(\A)\leq e-1$, and $\Bour(\A)=\binom e2$ if and only if $t(\A)=e-1$. In this case, every Bourbaki ideal associated with a minimal syzygy $\nu\in\syz(J_\A)_e$ has the minimal resolution
\begin{equation}
\label{eq:extremal-bourbaki-resolution}
0\longrightarrow R(-e)^{e-1}
\longrightarrow R(-(e-1))^e \longrightarrow I_\nu\longrightarrow0.
\end{equation}
\end{Theorem}

\begin{proof}
The assertion is clear if $\A$ is free. Assume that $\A$ is not free. By \eqref{eq:bourbaki-postulation-degrees},
\[
\Bour(\A)
=\binom e2-\dim_k(I_\nu)_{e-2}
\leq\binom e2.
\]
Moreover, $t(\A)\geq1$, while $\reg(I_\nu)\leq e-1$ implies $t(\A)\leq e-1$. Hence
\[
\Bour(\A)=\binom e2
\quad\Longleftrightarrow\quad (I_\nu)_{e-2}=0 \quad\Longleftrightarrow\quad t(\A)=e-1.
\]

Assume that these conditions hold. Since $\indeg(I_\nu)=e-1$ and $\reg(I_\nu)\leq e-1$, the ideal $I_\nu$ is generated in degree $e-1$. Moreover,
\[
\dim_k(I_\nu)_{e-1}
=\binom{e+1}{2}-\Bour(\A)
=e.
\]
Thus $I_\nu$ has exactly $e$ minimal generators, all of degree $e-1$. By the Hilbert--Burch theorem, its minimal resolution has the form
\[
0\longrightarrow\bigoplus_{j=1}^{e-1}R(-b_j)
\longrightarrow R(-(e-1))^e \longrightarrow I_\nu\longrightarrow0.
\]
Minimality gives $b_j\geq e$, while $\reg(I_\nu)\leq e-1$ gives $b_j\leq e$. Hence $b_j=e$ for every $j$, proving \eqref{eq:extremal-bourbaki-resolution}.

Finally, equality is attained for every $e\geq1$: by a triangle when
$e=1$, and, for $e\geq2$, by a generic arrangement of $e+2$ lines;
see Example~\ref{ex:generic-arrangements-bourbaki}.
\end{proof}

We set $\Delta(\A):=\binom{e_\A}{2}-\Bour(\A).$ By Theorem~\ref{thm:sharp-triangular-bourbaki-bound}, $\Delta(\A)\geq0$. We will say that $\A$ is \emph{extremal} if $\Delta(\A)=0$, or equivalently, if $\Bour(\A)=\binom{e_\A}{2}.$

We next refine the triangular bound using the syzygies associated with multiple points. A related construction using local derivations was considered in \cite[Proposition~5.1]{Dimca-local-derivations}. Together with Lemma~\ref{lem:regularity-bourbaki-ideal}, this gives bounds involving the two largest intersection multiplicities.

Let $m_1(\A)\geq m_2(\A)$ denote the two largest intersection multiplicities of $\A$. Thus there are distinct points $p_1,p_2\in L_2(\A)$ such that $m_{p_i}=m_i(\A)$ for $i=1,2$.
\begin{Theorem}
\label{thm:upper-bound-line-arrangements}
Let $\A$ be an essential nonfree arrangement of $n+1$ lines and set $e=e_\A$. Then
\[
\Bour(\A)\leq
\binom e2-\binom{m_2(\A)-1}{2},
\]
and
\begin{equation}
\label{eq:type-second-multiplicity}
t(\A)\leq e-m_2(\A)+1.
\end{equation}
If $e\neq n+1-m_1(\A)$, then
\[
\Bour(\A)\leq
\binom e2-\binom{m_1(\A)-1}{2},
\]
and
\begin{equation}
\label{eq:type-maximal-multiplicity}
t(\A)\leq e-m_1(\A)+1.
\end{equation}
\end{Theorem}
\begin{proof}
Let $\rho_i$ be the syzygy associated with $p_i$, $i=1,2$, and fix $0\neq\nu\in\syz(J_\A)_e$. We first show that at least one of $\rho_1,\rho_2$ does not belong to $R\nu$. Otherwise, $\rho_i=a_i\nu$ for homogeneous forms $a_i\in R$. Since the coordinates of $\rho_i$ have no nonconstant common factor, each $a_i$ must be a nonzero scalar. Hence $\rho_1$ and $\rho_2$ are nonzero scalar multiples, contradicting Proposition~\ref{prop:syzygies-associated-with-points}.

Choose $i\in\{1,2\}$ such that $\rho_i\notin R\nu$. Since $R\nu=\ker(\pi_\nu)$, one has $\pi_\nu(\rho_i)\neq0$. As $\deg(\rho_i)=n+1-m_{p_i}$, the Bourbaki sequence gives a nonzero form $u_i\in(I_\nu)_{e-m_{p_i}+1}.$ Therefore $t(\A)\leq e-m_{p_i}+1\leq e-m_2(\A)+1,$ which proves \eqref{eq:type-second-multiplicity}.

If $m_{p_i}=2$, the first bound for $\Bour(\A)$ follows from Theorem~\ref{thm:sharp-triangular-bourbaki-bound}. Assume that $m_{p_i}\geq3$. Multiplication by $u_i$ gives an injection
\[
R_{m_{p_i}-3}\longrightarrow(I_\nu)_{e-2},
\qquad h\longmapsto u_ih.
\]
Hence
\[
\dim_k(I_\nu)_{e-2} \geq\binom{m_{p_i}-1}{2} \geq\binom{m_2(\A)-1}{2}.
\]
By \eqref{eq:bourbaki-postulation-degrees},
\[
\Bour(\A)
\leq
\binom e2-\binom{m_2(\A)-1}{2}.
\]

Now assume that $e\neq n+1-m_1(\A)$. By \eqref{eq:initial-degree-multiplicity-bound}, $e<n+1-m_1(\A)=\deg(\rho_1)$. If $\rho_1\in R\nu$, then $\rho_1=a\nu$ for a homogeneous form $a$ of positive degree, so $a$ would be a nonconstant common factor of the coordinates of $\rho_1$, a contradiction. Hence $\rho_1\notin R\nu$. Applying the preceding argument to $p_1$ gives
\[
\Bour(\A)
\leq
\binom e2-\binom{m_1(\A)-1}{2}
\]
and \eqref{eq:type-maximal-multiplicity}.
\end{proof}

The behavior of the initial degree with respect to the maximal intersection multiplicity is already well understood. In particular, \cite[Theorem~1.10]{Dimca-local-derivations} gives strong restrictions in terms of $m_1(\A)$. Our estimates provide additional information involving the second largest multiplicity $m_2(\A)$ and, at the same time, control the type and the Bourbaki degree.

\begin{Corollary}
\label{cor:multiplicity-and-type-consequences}
Let $\A$ be an essential nonfree arrangement of $n+1$ lines and set $e=e_\A$. Then
\[
m_2(\A)\leq e.
\]
Moreover:
\begin{enumerate}
\item if $m_2(\A)=e$, then $\A$ is plus-one generated;
\item if $m_2(\A)\geq e-1$, then $t(\A)\leq2$;
\item if $e\neq n+1-m_1(\A)$ and $m_1(\A)=e$, then $\A$ is
plus-one generated.
\end{enumerate}
\end{Corollary}

\begin{proof}
Since $\A$ is not free, $t(\A)\geq1$. By \eqref{eq:type-second-multiplicity},
\[
1\leq t(\A)\leq e-m_2(\A)+1,
\]
and hence $m_2(\A)\leq e$.

If $m_2(\A)=e$, then \eqref{eq:type-second-multiplicity} gives $t(\A)=1$. Therefore $\A$ is plus-one generated by \cite[Theorem~1.1 and Definition~1.2]{Abe-Dimca-Pokora-hierarchy}.

If $m_2(\A)\geq e-1$, then again \eqref{eq:type-second-multiplicity} gives $t(\A)\leq2$.

Finally, if $e\neq n+1-m_1(\A)$ and $m_1(\A)=e$, then \eqref{eq:type-maximal-multiplicity} gives $t(\A)=1$, and hence $\A$ is plus-one generated.
\end{proof}

Combining the type with the multiplicity bounds gives a two-sided estimate for the Bourbaki degree. The lower extremal case also has a simple Hilbert--Burch resolution.

\begin{Proposition}
\label{prop:type-bourbaki-multiplicity-sandwich}
Let $\A$ be an essential nonfree arrangement of $n+1$ lines and set $e=e_\A$ and $t=t(\A)$. Then
\[
\binom{t+1}{2}
\leq
\Bour(\A)
\leq
\binom e2-\binom{m_2(\A)-1}{2}.
\]
Consequently,
\[
\binom{t+1}{2}+\binom{m_2(\A)-1}{2}
\leq\binom e2.
\]
If $e\neq n+1-m_1(\A)$, then $m_2(\A)$ may be replaced by $m_1(\A)$ in both inequalities.

Moreover,
\[
\Bour(\A)=\binom{t+1}{2}
\]
if and only if every Bourbaki ideal associated with a minimal syzygy $\nu\in\syz(J_\A)_e$ has the minimal resolution
\[
0\longrightarrow R(-t-1)^t
\longrightarrow R(-t)^{t+1} \longrightarrow I_\nu\longrightarrow0.
\]
\end{Proposition}

\begin{proof}
Since $(I_\nu)_j=0$ for $j<t$,
\[
\HF_{R/I_\nu}(t-1)
=\dim_kR_{t-1}
=\binom{t+1}{2}.
\]
As $R/I_\nu$ is a one-dimensional Cohen--Macaulay ring, its Hilbert function is nondecreasing and bounded above by its multiplicity. Therefore
\[
\binom{t+1}{2}
\leq\deg(R/I_\nu)
=\Bour(\A).
\]
The upper bounds follow from Theorem~\ref{thm:upper-bound-line-arrangements}.

Assume now that $\Bour(\A)=\binom{t+1}{2}$. Then $\HF_{R/I_\nu}(t-1)=\deg(R/I_\nu),$ so the Hilbert function is constant from degree $t-1$ onward. Hence $\reg(R/I_\nu)\leq t-1$ and $\reg(I_\nu)\leq t$. Since $\indeg(I_\nu)=t$, the ideal $I_\nu$ is generated in degree $t$, and
\[
\dim_k(I_\nu)_t
=\binom{t+2}{2}-\binom{t+1}{2}
=t+1.
\]
Thus Hilbert--Burch gives
\[
0\longrightarrow
\bigoplus_{j=1}^{t}R(-b_j) \longrightarrow R(-t)^{t+1} \longrightarrow I_\nu\longrightarrow0.
\]
Minimality gives $b_j\geq t+1$, while $\reg(I_\nu)\leq t$ gives $b_j\leq t+1$. Hence $b_j=t+1$ for every $j$.

Conversely, the displayed resolution gives $\deg(R/I_\nu)=\binom{t+1}{2}$, and therefore $\Bour(\A)=\binom{t+1}{2}$.
\end{proof}

The lattice class $L(n+1,m)$, consisting of arrangements with one $m$-fold point and all other intersection points double, was studied in \cite[Proposition~4.7]{Dimca-Ibadula-Macinic}. The extremal case of the triangular bound can be characterized precisely in terms of this class.

\begin{Theorem}
\label{thm:extremal-triangular-arrangements}
Let $\A$ be an essential nonfree arrangement of $n+1$ lines and set $e=e_\A$. Then the following conditions are equivalent:
\begin{enumerate}
\item[{\rm (a)}] $\Bour(\A)=\binom e2$;
\item[{\rm (b)}] $t(\A)=e-1$;
\item[{\rm (c)}] $m_2(\A)=2$;
\item[{\rm (d)}] $m_1(\A)=n+1-e$ and $m_2(\A)=2$.
\end{enumerate}
\end{Theorem}

\begin{proof}
The equivalence of (a) and (b) follows from Theorem~\ref{thm:sharp-triangular-bourbaki-bound}. Moreover, Theorem~\ref{thm:upper-bound-line-arrangements} gives
\[
\Bour(\A)=\binom e2
\quad\Longrightarrow\quad
\binom{m_2(\A)-1}{2}=0,
\]
and hence $m_2(\A)=2$. Thus (a) implies (c).

Assume now that $m_2(\A)=2$ and set $m=m_1(\A)$. Then every intersection point other than a point of multiplicity $m$ is double, so $b_2(\A) =\binom{n+1}{2}-\binom{m-1}{2}.$ We claim that $e=n+1-m$. Suppose otherwise. By \eqref{eq:initial-degree-multiplicity-bound}, $e<n+1-m$, while \eqref{eq:type-maximal-multiplicity} and $t(\A)\geq1$ give $m\leq e$. Set $q=n+1-e-m\geq1$. By Theorem~\ref{thm:bourbaki-characteristic-polynomial},
\[
\begin{aligned}
\Bour(\A)-\binom e2
&= e^2-ne+\binom{n+1}{2} -\binom{m-1}{2}-n-\binom e2\\ &=
\frac{q(2m+q-3)}{2}>0,
\end{aligned}
\]
contradicting Theorem~\ref{thm:sharp-triangular-bourbaki-bound}. Hence $e=n+1-m$, proving (d).

Finally, substituting $m=n+1-e$ in the preceding formula gives $\Bour(\A)=\binom e2$. Thus (d) implies (a), and all four conditions are equivalent.
\end{proof}

Theorem~\ref{thm:extremal-triangular-arrangements} characterizes the case in which there is at most one intersection point of multiplicity greater than two. We next consider the natural case of two such points. In this setting the refined multiplicity bound can again be computed exactly.

A point $p\in L_2(\A)$ is called \emph{modular} if, for every $q\in L_2(\A)$ with $q\neq p$, the line joining $p$ and $q$ belongs to $\A$. An essential line arrangement in $\PP^2$ is \emph{supersolvable} if it has a modular point; see, for instance, \cite{Dimca-Sticlaru-supersolvable}.

\begin{Proposition}
\label{prop:two-multiple-points}
Let $\A$ be an essential arrangement of $n+1$ lines having exactly two intersection points $p_1,p_2$ of multiplicity at least three, and assume $m_1:=m_{p_1}\geq m_2:=m_{p_2}$. All other intersection points are therefore double. Set $r:=n+1-m_1-m_2$. Then $r\geq-1$, and the following hold.

\begin{enumerate}
\item If $r=-1$, then $\A$ is supersolvable, and hence free.

\item If $r\geq0$, then $\A$ is not free and
$e_\A=n+1-m_1=m_2+r$. Moreover, $t(\A)\leq r+1$ and
\[
\Bour(\A)
=
\binom{e_\A}{2}-\binom{m_2-1}{2}
=
\frac{(r+1)(2m_2+r-2)}{2},
\]
\item If $r=0$, then $\A$ is plus-one generated, $\Bour(\A)=m_2-1$ and its exponents are $(m_2,m_1,n-1)$.
\end{enumerate}
\end{Proposition}
\begin{proof}
The sets of lines through $p_1$ and $p_2$ have respectively $m_1$ and $m_2$ elements and have at most one line in common. Hence $n+1\geq m_1+m_2-1$, so $r\geq-1$.

Suppose first that $r=-1$. Then the line $p_1p_2$ belongs to $\A$ and every line of $\A$ passes through either $p_1$ or $p_2$. Thus $p_1$ is a modular point, so $\A$ is supersolvable and hence free.

Assume now that $r\geq0$. Since $p_1$ and $p_2$ are the only intersection points of multiplicity greater than two,
\[
b_2(\A)
=
\binom{n+1}{2}
-\binom{m_1-1}{2} -\binom{m_2-1}{2}.
\]
We first show that $\A$ is not free. Set $E=n+1-m_1=m_2+r$. By \eqref{eq:initial-degree-multiplicity-bound}, $e_\A\leq E$. For $1\leq s\leq E$, write $u=E-s$, so $u\geq0$. Using the preceding formula for $b_2(\A)$, we obtain
\[
\overline{\chi}_\A(s)
=
\frac{(r+1)(2m_2+r-2)}{2}
+ u(m_1-m_2-r+u-1).
\]
Since $m_1\geq m_2$ and $u\geq0$,
\[
\overline{\chi}_\A(s)
\geq
\frac{(r+1)(2m_2+r-2)}{2}
+ u(u-r-1).
\]
Since
\[
u(u-r-1)\geq-\frac{(r+1)^2}{4},
\]
it follows that
\[
\overline{\chi}_\A(s)
\geq
\frac{(r+1)(4m_2+r-5)}{4}>0,
\]
where the last inequality follows from $m_2\geq3$ and $r\geq0$. In particular, $\overline{\chi}_\A(e_\A)>0$, so Theorem~\ref{thm:bourbaki-characteristic-polynomial} shows that $\A$ is not free.

We next show that $e_\A=n+1-m_1$. Otherwise, \eqref{eq:initial-degree-multiplicity-bound} gives $e_\A<n+1-m_1$, while \eqref{eq:type-maximal-multiplicity} and $t(\A)\geq1$ give $m_1\leq e_\A$. Set $q=n+1-e_\A-m_1\geq1$. Using the formula for $b_2(\A)$ and Theorem~\ref{thm:bourbaki-characteristic-polynomial}, we obtain
\[
\Bour(\A)
-
\left(
\binom{e_\A}{2}-\binom{m_2-1}{2}
\right)
=
\frac{q(2m_1+q-3)}{2}>0,
\]
contradicting Theorem~\ref{thm:upper-bound-line-arrangements}. Therefore $e_\A=n+1-m_1=m_2+r.$ Substitution in $\Bour(\A)=\overline{\chi}_\A(e_\A)$ now gives
\[
\Bour(\A)
=
\binom{e_\A}{2}-\binom{m_2-1}{2}
=
\frac{(r+1)(2m_2+r-2)}{2},
\]
while \eqref{eq:type-second-multiplicity} gives $t(\A)\leq e_\A-m_2+1=r+1$.

Finally, suppose that $r=0$. Then $e_\A=m_2$, and hence $t(\A)\leq1$. Since $\A$ is not free, $t(\A)\geq1$, so $t(\A)=1$ and $\A$ is plus-one generated. If $(d_1,d_2,d_3)$ are its exponents, then $d_1=e_\A=m_2$ and $d_2=n+1-d_1=n+1-m_2=m_1.$ Moreover, $\Bour(\A)=m_2-1=d_3-d_2+1$, and therefore $d_3=m_1+m_2-2=n-1.$ Thus the exponents are $(m_2,m_1,n-1)$.
\end{proof}

\begin{Remark}
\label{rem:two-multiple-points}
Proposition~\ref{prop:two-multiple-points} may be viewed as a two-point extension of Theorem~\ref{thm:extremal-triangular-arrangements}. Indeed, when there is only one intersection point of multiplicity greater than two, the extremal value $\Bour(\A)=\binom{e_\A}{2}$ is attained. With two such points, the refined bound is attained:
\[
\Bour(\A)=\binom{e_\A}{2}-\binom{m_2-1}{2}.
\]

Moreover, the integer $r=n+1-m_1-m_2$ measures the difference between the initial degree and the second largest multiplicity, since $e_\A=m_2+r,\qquad t(\A)\leq r+1.$ Thus the geometry of the two distinguished multiple points gives direct control over the type.

The case $r=0$ is particularly simple. Then
\[
e_\A=m_2,\qquad t(\A)=1,\qquad \Bour(\A)=m_2-1,
\]
so $\A$ is plus-one generated.
\end{Remark}
The cases $e_\A\leq2$ were classified in
\cite[Theorem~4.12 and Corollary~4.13]{Dimca-Ibadula-Macinic}, while
the nonfree cases with $e_\A=3$ needed below are covered by
Burity and Toh\u{a}neanu \cite[Theorem~2.6]{Burity-Tohaneanu}.
The next theorem expresses these low-degree cases from the Bourbaki
viewpoint.

\begin{Theorem}
\label{thm:low-degree-bourbaki-classification}
Let $\A$ be an essential arrangement of $n+1$ lines.
\begin{enumerate}
\item If $e_\A=1$, then $\A$ is free.

\item If $e_\A=2$, then $\Bour(\A)\in\{0,1\}$. Thus $\A$ is either
free or nearly free.

\item Assume that $e_\A=3$. Then $0\leq\Bour(\A)\leq3$. If $\A$ is
not free, exactly one of the following occurs:
\begin{enumerate}
\item $m_1(\A)=n-2$. Fix a point $p\in L_2(\A)$ with
$m_p=n-2$, and let $s$ be the number of triple points distinct from $p$. Then $s\in\{0,1,2\}$ and
\[
\Bour(\A)=3-s.
\]
More precisely, if $s=0$, then $\Bour(\A)=3$ and $t(\A)=2$; if $s=1$, then $\Bour(\A)=2$ and $\A$ is plus-one generated but not nearly free; and if $s=2$, then $\Bour(\A)=1$ and $\A$ is nearly free.

\item $n=6$ and $m_1(\A)=3$. For $j\geq2$, let $n_j(\A):=\#\{p\in L_2(\A):m_p=j\}$. In this case $\A$ has only double and
triple points, and either
\[
n_3(\A)=4,\qquad \Bour(\A)=2,
\]
or
\[
n_3(\A)=5,\qquad \Bour(\A)=1.
\]
In the first case $\A$ is plus-one generated but not nearly free, while in the second case $\A$ is nearly free.
\end{enumerate}
\end{enumerate}
\end{Theorem}

\begin{proof}
Assume first that $\A$ is not free. Then $\Bour(\A)\geq1$, while Theorem~\ref{thm:sharp-triangular-bourbaki-bound} gives $\Bour(\A)\leq\binom{e_\A}{2}.$ For $e_\A=1$ this is impossible, proving (1). If $e_\A=2$, then $\Bour(\A)=1$, so $\A$ is nearly free. This proves (2).

Assume now that $e_\A=3$ and that $\A$ is not free. By \cite[Theorem~2.6]{Burity-Tohaneanu}, the nonfree members of the
classification with minimal degree three fall into exactly two
possibilities:
\[
m_1(\A)=n-2
\qquad\text{or}\qquad
n=6,\quad m_1(\A)=3.
\]

Suppose first that $m_1(\A)=n-2$, and choose $p\in L_2(\A)$ with $m_p=n-2$. Since $e_\A=3$ and $\A$ is nonfree, Corollary~\ref{cor:multiplicity-and-type-consequences} gives $m_2(\A)\leq3$. Hence every intersection point distinct from $p$ is double or triple. Let $s$ be the number of triple points distinct from $p$, and let $N_2$ be the number of double points distinct from $p$. Counting pairs of lines gives $\binom{n+1}{2} = \binom{n-2}{2}+3s+N_2,$ while $b_2(\A)=n-3+2s+N_2.$ Eliminating $N_2$ yields $b_2(\A)=4n-6-s$. Therefore Theorem~\ref{thm:bourbaki-characteristic-polynomial} gives $\Bour(\A) =\overline{\chi}_\A(3) =3-s.$ Since $\Bour(\A)\geq1$, it follows that $s\in\{0,1,2\}$.

If $s=0$, then $\Bour(\A)=3=\binom32$, and Theorem~\ref{thm:sharp-triangular-bourbaki-bound} gives $t(\A)=2$. If $s=1$, then $\Bour(\A)=2$, and Proposition~\ref{prop:type-bourbaki-multiplicity-sandwich} gives $\binom{t(\A)+1}{2}\leq2$. Hence $t(\A)=1$, so $\A$ is plus-one generated; it is not nearly free since $\Bour(\A)\neq1$. If $s=2$, then $\Bour(\A)=1$, so $\A$ is nearly free.

It remains to consider $n=6$ and $m_1(\A)=3$. Then every intersection point is double or triple. Writing $n_i=n_i(\A)$, pair counting gives $n_2+3n_3=\binom72=21,$ and hence $b_2(\A)=n_2+2n_3=21-n_3$. Thus $\Bour(\A) =\overline{\chi}_\A(3) =6-n_3.$ Since $1\leq\Bour(\A)\leq3$, one has $3\leq n_3\leq5$. If $n_3=3$, then $\Bour(\A)=3$, and Theorem~\ref{thm:extremal-triangular-arrangements} would give $m_1(\A)=n+1-e_\A=4$, a contradiction. Hence $n_3\in\{4,5\}$.

If $n_3=4$, then $\Bour(\A)=2$, and the preceding argument gives $t(\A)=1$, so $\A$ is plus-one generated but not nearly free. If $n_3=5$, then $\Bour(\A)=1$, so $\A$ is nearly free.
\end{proof}

We conclude this subsection by translating the preceding Bourbaki
bounds into lower bounds for the global Tjurina number and comparing
them with the classical bounds of du Plessis and Wall. Over $\CC$, for
a reduced plane curve of degree $d$ whose syzygy module has initial
degree $e$, du Plessis and Wall proved
\cite[Theorem~3.2]{DuPlessis-Wall}
\[
(d-1)(d-e-1) \leq\tau \leq (d-1)(d-e-1)+e^2.
\]
Moreover, if $2e+1>d$, their refined upper bound is
\[
\tau
\leq (d-1)(d-e-1)+e^2 -\frac{(2e+1-d)(2e+2-d)}{2}.
\]
For an arrangement $\A$ of $n+1$ lines, these become $n(n-e)\leq\tau(\A)\leq n(n-e)+e^2,$ and, when $e>n/2$,
\[
\tau(\A)
\leq n(n-e)+e^2-\binom{2e+1-n}{2}.
\]

On the other hand, rewriting \eqref{eq:bourbaki-tjurina-arrangement}, we have
\begin{equation}
\label{eq:tjurina-bourbaki-arrangement}
\tau(\A)=n(n-e)+e^2-\Bour(\A).
\end{equation}
Thus the lower and upper du Plessis--Wall bounds correspond, respectively, to upper and lower bounds for the Bourbaki degree. In particular, the general estimate $\Bour(\A)\leq e^2$ recovers the lower du Plessis--Wall bound, while Theorem~\ref{thm:sharp-triangular-bourbaki-bound} gives a substantially stronger estimate for line arrangements.

\begin{Corollary}
\label{cor:tjurina-lower-bounds-line-arrangements}
Let $\A$ be an essential arrangement of $n+1$ lines and set $e=e_\A$. Then
\[
\tau(\A)\geq
n(n-e)+\binom{e+1}{2}.
\]
The correction term $\binom{e+1}{2}$ is optimal among correction terms depending only on $e$.

If $\A$ is not free, then
\begin{equation}
\label{eq:refined-second-multiplicity-tjurina-bound}
\tau(\A)\geq
n(n-e) +\binom{e+1}{2} +\binom{m_2(\A)-1}{2}.
\end{equation}
If, in addition, $e\neq n+1-m_1(\A)$, then
\begin{equation}
\label{eq:refined-maximal-multiplicity-tjurina-bound}
\tau(\A)\geq
n(n-e) +\binom{e+1}{2} +\binom{m_1(\A)-1}{2}.
\end{equation}
\end{Corollary}

\begin{proof}
By \eqref{eq:tjurina-bourbaki-arrangement} and Theorem~\ref{thm:sharp-triangular-bourbaki-bound},
\[
\tau(\A)
\geq n(n-e)+e^2-\binom e2
=
n(n-e)+\binom{e+1}{2}.
\]
The refined estimates follow in the same way from Theorem~\ref{thm:upper-bound-line-arrangements}.

For optimality, equality is attained by a triangle when $e=1$ and,
for $e\geq2$, by a generic arrangement of $e+2$ lines; see
Example~\ref{ex:generic-arrangements-bourbaki}. Hence the correction
term cannot be increased as a function of $e$ alone.
\end{proof}

\begin{Remark}
\label{rem:comparison-dimca-tjurina}
The multiplicity corrections above also improve the bounds obtained by Dimca in \cite[Theorem~1.3]{Dimca-minimal-Tjurina}. In the notation of that result,
\[
e=r,\qquad
m_1(\A)=m(C),\qquad
m_2(\A)=n(C).
\]
For an arrangement of $n+1$ lines, \cite[Theorem~1.3(i)]{Dimca-minimal-Tjurina} gives
\[
\tau(\A)\geq
n(n-e)+\binom e2+\binom{m_2(\A)}2+1.
\]
The difference between the correction term in \eqref{eq:refined-second-multiplicity-tjurina-bound} and the one above is
\[
\binom{e+1}{2}
+\binom{m_2(\A)-1}{2} -\binom e2-\binom{m_2(\A)}2-1
=
e-m_2(\A).
\]
Since $m_2(\A)\leq e$, our estimate is at least as strong, and is strictly stronger whenever $m_2(\A)<e$.

Likewise, if $e\neq n+1-m_1(\A)$, the improvement provided by \eqref{eq:refined-maximal-multiplicity-tjurina-bound} over \cite[Theorem~1.3(ii)]{Dimca-minimal-Tjurina} is exactly $e-m_1(\A)\geq0$.

In particular, if $e\geq3$ and $m_2(\A)\geq3$, then
\[
\tau(\A)\geq
n(n-e) +\binom{e+1}{2} +\binom{m_2(\A)-1}{2} \geq n(n-e)+7,
\]
recovering \cite[Corollary~1.4]{Dimca-minimal-Tjurina}. The full correction above is stronger: compared with Dimca's correction $\binom e2+4$, its improvement is
\[
e+\binom{m_2(\A)-1}{2}-4.
\]

Recently, Pokora \cite[Corollary~4.2]{Pokora-Tjurina-bounds}
proved that if $\A$ is a free arrangement of $n+1$ lines, then
\[
\tau(\A)\geq\frac{3n^2}{4}.
\]
This bound is of a different nature from the estimates above, since it
depends only on the number of lines and assumes freeness. Indeed, if
$\A$ is free with exponents $(e,n-e)$, then
$\Bour(\A)=0$, and \eqref{eq:tjurina-bourbaki-arrangement} gives
\[
\tau(\A)
=n^2-e(n-e)
=
\frac{3n^2}{4}
+\frac{(n-2e)^2}{4}.
\]
Thus Pokora's bound is recovered in the free case, with equality
precisely when the two exponents are equal. By contrast, the bounds
\eqref{eq:refined-second-multiplicity-tjurina-bound} and
\eqref{eq:refined-maximal-multiplicity-tjurina-bound} concern
nonfree arrangements and incorporate the initial degree together with
the two largest intersection multiplicities.
\end{Remark}

Combining the refined du Plessis--Wall upper bound with Theorem~\ref{thm:upper-bound-line-arrangements} also gives a restriction in the large-$e$ range.

\begin{Corollary}
\label{cor:large-e-bourbaki-window}
Assume that $k=\CC$. Let $\A$ be an essential nonfree arrangement of $n+1$ lines and assume that $e>n/2$. Then
\[
\binom{2e+1-n}{2}
\leq
\Bour(\A)
\leq
\binom e2-\binom{m_2(\A)-1}{2}.
\]
Consequently,
\[
\binom{2e+1-n}{2}
+\binom{m_2(\A)-1}{2} \leq
\binom e2.
\]
If $e\neq n+1-m_1(\A)$, then $m_2(\A)$ may be replaced by $m_1(\A)$ in both inequalities.
\end{Corollary}

\begin{proof}
The refined du Plessis--Wall upper bound gives
\[
\tau(\A)
\leq n(n-e)+e^2-\binom{2e+1-n}{2}.
\]
By \eqref{eq:tjurina-bourbaki-arrangement}, this is equivalent to $\Bour(\A)\geq\binom{2e+1-n}{2}.$ The upper bound follows from Theorem~\ref{thm:upper-bound-line-arrangements}, and combining the two gives the second inequality. The assertion involving $m_1(\A)$ follows from the corresponding part of the same theorem.
\end{proof}

\subsection{Freeness defect and Terao's conjecture}
\label{subsec:freeness-defect-terao}

Set $N(\A):=J_\A^{\mathrm{sat}}/J_\A$ and define the \emph{freeness defect} of $\A$ by $\nu(\A):=\max_{q\in\ZZ}\dim_kN(\A)_q$. The formulas for this invariant are due to Dimca; see \cite[Theorem~1.7]{Dimca-minimal-Tjurina}. In the first range, its relation with the Bourbaki degree was also recorded in \cite[formula~(1.11)]{Abe-Dimca-Pokora-hierarchy}. The following proposition gives both ranges in our notation.

\begin{Proposition}
\label{prop:bourbaki-freeness-defect}
Assume that $k=\CC$. Let $\A$ be an essential arrangement of $n+1$ lines and set $e=e_\A$.
\begin{enumerate}
\item[{\rm (a)}] If $e\leq\lfloor n/2\rfloor$, then
\begin{equation}
\label{eq:bourbaki-equals-freeness-defect}
\nu(\A)=\Bour(\A).
\end{equation}

\item[{\rm (b)}] If $e\geq\lfloor n/2\rfloor$, then
\[
\nu(\A)=\Bour(\A)-\kappa(n,e),
\]
where
\[
\kappa(n,e):=
\left\lfloor\frac{n^2}{4}\right\rfloor-e(n-e)
=
\left(e-\left\lfloor\frac n2\right\rfloor\right)
\left(e-\left\lceil\frac n2\right\rceil\right).
\]
\end{enumerate}
At $e=\lfloor n/2\rfloor$, one has $\kappa(n,e)=0$, so the two formulas agree.
\end{Proposition}

\begin{proof}
If $e<(n+1)/2$, then \cite[Theorem~1.7(1)]{Dimca-minimal-Tjurina} gives
\[
\nu(\A)=n^2-e(n-e)-\tau(\A).
\]
By \eqref{eq:tjurina-bourbaki-arrangement}, the right-hand side is $\Bour(\A)$. Since $e$ is an integer, the condition $e<(n+1)/2$ is equivalent to $e\leq\lfloor n/2\rfloor$, proving (a).

If $e\geq(n-1)/2$, then \cite[Theorem~1.7(2)]{Dimca-minimal-Tjurina} gives
\[
\nu(\A)
=
\left\lceil\frac{3n^2}{4}\right\rceil-\tau(\A).
\]
Hence
\[
\Bour(\A)-\nu(\A)
=
\left\lfloor\frac{n^2}{4}\right\rfloor-e(n-e)
=
\kappa(n,e).
\]
Since $e$ is an integer, $e\geq(n-1)/2$ is equivalent to $e\geq\lfloor n/2\rfloor$. Finally,
\[
\left\lfloor\frac{n^2}{4}\right\rfloor
=
\left\lfloor\frac n2\right\rfloor
\left\lceil\frac n2\right\rceil,
\]
which gives the product expression for $\kappa(n,e)$.
\end{proof}

Combining Proposition~\ref{prop:bourbaki-freeness-defect} with the multiplicity bounds obtained above gives a corresponding refinement for the freeness defect.

\begin{Corollary}
\label{cor:refined-freeness-defect-bound}
Assume that $k=\CC$. Let $\A$ be an essential nonfree arrangement and set $e=e_\A$. If $e\leq\lfloor n/2\rfloor$, then
\begin{equation}
\label{eq:refined-freeness-defect-second-multiplicity}
\nu(\A)
\leq
\binom e2-\binom{m_2(\A)-1}{2}.
\end{equation}
If, in addition, $e\neq n+1-m_1(\A)$, then
\[
\nu(\A)
\leq
\binom e2-\binom{m_1(\A)-1}{2}.
\]
Still assuming $e\leq\lfloor n/2\rfloor$, one has
\[
\nu(\A)=\binom e2
\quad\text{if and only if}\quad
m_2(\A)=2
\quad\text{if and only if}\quad t(\A)=e-1.
\]

If $e\geq\lfloor n/2\rfloor$, then
\[
\nu(\A)
\leq
\binom e2-\binom{m_2(\A)-1}{2}-\kappa(n,e),
\]
with $m_2(\A)$ replaced by $m_1(\A)$ when $e\neq n+1-m_1(\A)$.
\end{Corollary}

\begin{proof}
The assertions follow immediately from Proposition~\ref{prop:bourbaki-freeness-defect}, Theorem~\ref{thm:upper-bound-line-arrangements}, and Theorem~\ref{thm:extremal-triangular-arrangements}.
\end{proof}

\begin{Remark}
\label{rem:freeness-defect-comparison}
In the range $e\leq\lfloor n/2\rfloor$, the preceding estimate improves \cite[Corollary~1.8]{Dimca-minimal-Tjurina}. Indeed, Dimca's bound becomes
\[
\nu(\A)
\leq
\binom{e+1}{2}-\binom{m_2(\A)}2-1.
\]
The difference between this upper bound and \eqref{eq:refined-freeness-defect-second-multiplicity} is $e-m_2(\A)\geq0$. Thus our estimate is strictly stronger whenever $m_2(\A)<e$ and agrees with Dimca's estimate when $m_2(\A)=e$.
\end{Remark}

We now turn to Terao's conjecture, which asserts that freeness of an arrangement is determined by its intersection lattice. The refined Bourbaki bound gives the following criterion.

\begin{Theorem}
\label{thm:terao-triangular-criterion}
Assume that $k=\CC$. Let $\A$ be an essential free arrangement of $n+1$ lines with exponents $(e,n-e)$, where $e\leq n-e$, and let $\B$ be an arrangement such that $L(\B)\simeq L(\A)$. Set $\mu:=m_2(\A)=m_2(\B)$. If
\begin{equation}
\label{eq:terao-triangular-condition}
\binom{e+1}{2}
< n+\binom{\mu-1}{2},
\end{equation}
then $\B$ is free, necessarily with exponents $(e,n-e)$.

If, in addition, $e\neq n+1-m_1(\A)$, the same conclusion holds under the condition
\[
\binom{e+1}{2}
< n+\binom{m_1(\A)-1}{2}.
\]
\end{Theorem}

\begin{proof}
Set $s=e_\B$ and $u=n-e$. Since $\A$ is free and $L(\B)\simeq L(\A)$,
\[
\overline{\chi}_\B(t)
=
\overline{\chi}_\A(t)
=
(t-e)(t-u).
\]
Suppose that $\B$ is not free. By Theorem~\ref{thm:bourbaki-characteristic-polynomial}, $\Bour(\B)=(s-e)(s-u)>0.$ Since $e\leq u$, either $s<e$ or $s>u$.

We first show that $s>u$ is impossible. Write $h=u-e=n-2e\geq0$ and $q=s-u\geq1$. Then $s>n/2$, and the refined du Plessis--Wall upper bound is, by \eqref{eq:tjurina-bourbaki-arrangement}, equivalent to
\[
\Bour(\B)\geq\binom{2s+1-n}{2}.
\]
On the other hand,
\[
\Bour(\B)=q(h+q),
\qquad
\binom{2s+1-n}{2}
=
\binom{h+2q+1}{2}.
\]
But
\[
\binom{h+2q+1}{2}-q(h+q)
=
\frac{h^2+2hq+h+2q^2+2q}{2}>0,
\]
a contradiction. Hence $s<e$.

Since $\B$ is not free, Theorem~\ref{thm:upper-bound-line-arrangements} gives $\Bour(\B) \leq \binom s2-\binom{\mu-1}{2}.$ Set $c=\binom{\mu-1}{2}$. Since $s\leq e-1$,
\[
(e-s)(u-s)-\binom s2+c
=
n+c-\binom{e+1}{2} +\frac{(e-s-1)(2n-e-s)}{2}.
\]
The last term is nonnegative, while \eqref{eq:terao-triangular-condition} gives $n+c-\binom{e+1}{2}>0$. Therefore $\Bour(\B)=(e-s)(u-s)>\binom s2-c,$ contradicting the preceding upper bound. Thus $\B$ is free. Its exponents are determined by the common characteristic polynomial, so they are $(e,n-e)$.

Finally, suppose that $e\neq n+1-m_1(\A)$. By \eqref{eq:initial-degree-multiplicity-bound}, $e<n+1-m_1(\A)$. Since $s<e$ and $m_1(\B)=m_1(\A)$, one has $s\neq n+1-m_1(\B)$. Hence the maximal-multiplicity bound in Theorem~\ref{thm:upper-bound-line-arrangements} applies to $\B$. Replacing $\mu$ by $m_1(\A)$ in the preceding argument proves the second assertion.
\end{proof}

The preceding theorem yields a uniform criterion depending only on the number of lines and the smaller exponent.

\begin{Corollary}
\label{cor:uniform-terao-triangular-criterion}
Assume that $k=\CC$. Let $\A$ be an essential free arrangement of $n+1$ lines with exponents $(e,n-e)$, where $e\leq n-e$. If $\binom{e+1}{2}<n+1$, then Terao's conjecture holds for $\A$.
\end{Corollary}

\begin{proof}
If $m_2(\A)=2$, then the intersection lattice has type $L(n+1,m_1(\A))$, and freeness is determined by the lattice; see \cite[Proposition~4.7]{Dimca-Ibadula-Macinic}.

Assume that $m_2(\A)\geq3$. Then $\binom{m_2(\A)-1}{2}\geq1$, and hence
\[
\binom{e+1}{2}
< n+1 \leq n+\binom{m_2(\A)-1}{2}.
\]
The conclusion follows from Theorem~\ref{thm:terao-triangular-criterion}.
\end{proof}

\begin{Remark}
\label{rem:comparison-dimca-terao}
In our notation, Dimca's criterion
\cite[Corollary~3.1]{Dimca-minimal-Tjurina} is equivalent to
\[
n\geq\binom{e+1}{2}+e-4,
\]
whereas Corollary~\ref{cor:uniform-terao-triangular-criterion}
requires only
\[
n\geq\binom{e+1}{2}.
\]
Thus the two criteria agree for $e=4$, while for every $e\geq5$
our criterion lowers the required value of $n$ by exactly $e-4$.
The first strict improvement occurs for $e=5$ and $n=15$,
corresponding to arrangements of sixteen lines.

This improvement reflects the sharper Tjurina bound obtained above:
the difference between our multiplicity correction and Dimca's
corresponding correction is exactly
\[
e-m_2(\A)\geq0;
\]
see Remark~\ref{rem:comparison-dimca-tjurina}. More importantly,
Theorem~\ref{thm:terao-triangular-criterion} retains the
lattice-theoretic multiplicity $m_2(\A)$ and, under the additional
hypothesis $e\neq n+1-m_1(\A)$, also gives a criterion involving
$m_1(\A)$. Hence it can apply even when neither uniform numerical
condition is satisfied.
\end{Remark}

\section{The Bourbaki degree under addition and deletion}
\label{sec:addition-deletion}

The behavior of the initial degree under addition and deletion is well understood by the results of Abe, Dimca and Sticlaru \cite{Abe-Dimca-Stic}; related addition--deletion results for free and plus-one generated curves were obtained in \cite{Macinic-Pokora-addition}. Here we study instead the variation of the Bourbaki degree. We obtain exact addition--deletion formulas and show that, although the Bourbaki degree itself is not monotone, the defect $\Delta(\A)$ is monotone under addition. The latter admits a geometric interpretation in terms of the intersections of the original arrangement that coalesce on the added line.

\subsection{Addition--deletion and the defect}
\label{subsec:deletion-line}

Let $\A$ be an essential arrangement of $n+1$ lines, let $H\in\A$, and set $\A'=\A\setminus\{H\}$ and $\A^H=\{H\cap L:L\in\A'\}$. Assume that $\A'$ is essential, and write $s=|\A^H|$, $e=e_\A$, and $e'=e_{\A'}$. By \cite[Proposition~2.12(1)]{Abe-Dimca-Stic}, one has $e'\leq e\leq e'+1$.

A direct count gives
\[
b_2(\A)-b_2(\A')=s,
\qquad
\tau(\A)-\tau(\A')=2n-s,
\]
and therefore
\begin{equation}
\label{eq:reduced-characteristic-deletion}
\overline{\chi}_\A(t)
=
\overline{\chi}_{\A'}(t)-t+s-1.
\end{equation}

\begin{Proposition}
\label{prop:bourbaki-deletion-formula}
With the notation above,
\begin{equation}
\label{eq:general-bourbaki-deletion-formula}
\Bour(\A)-\Bour(\A')
=
s-e-1+(e-e')(e+e'-n+1).
\end{equation}
More precisely:
\begin{enumerate}
\item if $e=e'$, then
\[
\Bour(\A)-\Bour(\A')=s-e-1\leq0,
\qquad
\Delta(\A)-\Delta(\A')=e+1-s\geq0;
\]

\item if $e=e'+1$, then
\[
\Bour(\A)-\Bour(\A')=s+e-n-1\geq0,
\qquad
\Delta(\A)-\Delta(\A')=n-s\geq0.
\]
\end{enumerate}
Consequently, $\Delta(\A)\geq\Delta(\A')$. In particular, extremality is inherited by every essential deletion. More precisely, $\A$ is extremal if and only if $\A'$ is extremal and either
\[
e=e',\qquad s=e+1,
\]
or
\[
e=e'+1,\qquad s=n.
\]
\end{Proposition}

\begin{proof}
By Theorem~\ref{thm:bourbaki-characteristic-polynomial} and \eqref{eq:reduced-characteristic-deletion},
\[
\begin{aligned}
\Bour(\A)-\Bour(\A')
&=
\overline{\chi}_\A(e)-\overline{\chi}_{\A'}(e')\\
&=
\overline{\chi}_{\A'}(e)-\overline{\chi}_{\A'}(e')-e+s-1.
\end{aligned}
\]
Since $\overline{\chi}_{\A'}(t)=t^2-(n-1)t+c$ for some integer $c$,
\[
\overline{\chi}_{\A'}(e)-\overline{\chi}_{\A'}(e')
=
(e-e')(e+e'-n+1),
\]
which proves \eqref{eq:general-bourbaki-deletion-formula}.

If $e=e'$, then $\Bour(\A)-\Bour(\A')=s-e-1$. Moreover, $s\geq e'+2$ would imply $e=e'+1$ by \cite[Theorem~3.3]{Abe-Dimca-Stic}; hence $s\leq e+1$.

If $e=e'+1$, then $\Bour(\A)-\Bour(\A')=s+e-n-1$. If $s<n+1-e$, then $n-s>e-1=e'$, and \cite[Proposition~2.12(3)]{Abe-Dimca-Stic} would imply $e=e'$, a contradiction. Thus $s\geq n+1-e$.

The two formulas for $\Delta$ follow immediately, and their right-hand sides are nonnegative. If $\A$ is extremal, then $0\leq\Delta(\A')\leq\Delta(\A)=0$, so $\A'$ is also extremal and the corresponding increment vanishes. The converse follows from the same formulas.
\end{proof}

The preceding monotonicity has an additive form along a filtration. Choose three nonconcurrent lines of $\A$ and order the remaining lines so that
\[
\A_3\subset\A_4\subset\cdots\subset\A_{n+1}=\A,
\qquad |\A_j|=j,
\]
where $\A_3$ is a triangle. Write $\A_j=\A_{j-1}\cup\{H_j\}$, $s_j=|\A_j^{H_j}|$, and $e_j=e_{\A_j}$.

\begin{Corollary}
\label{cor:triangular-defect-filtration}
With the notation above,
\[
\Delta(\A)
=
\sum_{\substack{4\leq j\leq n+1\\ e_j=e_{j-1}}}
(e_j+1-s_j) +
\sum_{\substack{4\leq j\leq n+1\\ e_j=e_{j-1}+1}}
(j-1-s_j).
\]
Every summand is nonnegative. Equivalently,
\[
\Bour(\A)
=
\binom{e_\A}{2}
-
\sum_{\substack{4\leq j\leq n+1\\ e_j=e_{j-1}}}
(e_j+1-s_j) -
\sum_{\substack{4\leq j\leq n+1\\ e_j=e_{j-1}+1}}
(j-1-s_j).
\]
Hence $\A$ is extremal if and only if every step of the filtration preserves $\Delta$.
\end{Corollary}

\begin{proof}
Since $\A_3$ is a triangle, $e_3=1$ and $\Delta(\A_3)=0$. Applying Proposition~\ref{prop:bourbaki-deletion-formula} at each step gives
\[
\Delta(\A_j)-\Delta(\A_{j-1})
=
\begin{cases}
e_j+1-s_j,&e_j=e_{j-1},\\
j-1-s_j,&e_j=e_{j-1}+1.
\end{cases}
\]
Summing proves the formulas and the equality characterization.
\end{proof}

Thus the Bourbaki degree itself need not be monotone under addition or deletion, while $\Delta(\A)$ is always weakly increasing under addition. The preceding formula expresses its total increase as a sum of explicit nonnegative local contributions.

\subsection{Addition of a line and collision defects}
\label{subsec:addition-line}

We now reinterpret the preceding formulas from the constructive viewpoint. Let $\A'$ be an essential arrangement of $n$ lines, let $H\notin\A'$, and set $\A=\A'\cup\{H\}$. Write $s=|\A^H|$, $e'=e_{\A'}$, and $e=e_\A$.

We measure the intersections of $\A'$ that coalesce on $H$ by the \emph{collision number}
\[
c_H:=n-s.
\]
If $\A^H=\{p_1,\ldots,p_s\}$ and $r_i$ is the number of lines of $\A'$ through $p_i$, then $c_H=\sum_{i=1}^s(r_i-1)$. Thus an old point of multiplicity $r_i$ lying on $H$ contributes $r_i-1$ collisions. In particular, $c_H=0$ precisely when $H$ contains no intersection point of $\A'$.

\begin{Proposition}
\label{prop:bourbaki-addition-collision}
With the notation above,
\[
\Bour(\A)-\Bour(\A')
=
\begin{cases}
n-e'-1-c_H,&e=e',\\ e'-c_H,&e=e'+1,
\end{cases}
\]
and
\begin{equation}
\label{eq:triangular-defect-addition-collision}
\Delta(\A)-\Delta(\A')
=
\begin{cases}
c_H-(n-e'-1),&e=e',\\
c_H,&e=e'+1.
\end{cases}
\end{equation}
Consequently,
\[
e=e'\Longrightarrow c_H\geq n-e'-1, \qquad e=e'+1\Longrightarrow c_H\leq e'.
\]
In particular,
\[
c_H\leq e'
\qquad\text{or}\qquad
c_H\geq n-e'-1.
\]
Hence, if $n\geq2e'+3$, no added line can satisfy
\[
e'+1\leq c_H\leq n-e'-2.
\]
\end{Proposition}

\begin{proof}
Since $s=n-c_H$, Proposition~\ref{prop:bourbaki-deletion-formula} gives
\[
\Bour(\A)-\Bour(\A')
=
\begin{cases}
n-e'-1-c_H,&e=e',\\ e'-c_H,&e=e'+1.
\end{cases}
\]
The formula for $\Delta$ follows in the same way. Its nonnegativity gives the stated restrictions on $c_H$, and the gap is an immediate consequence.
\end{proof}

In the jumping case, formula \eqref{eq:triangular-defect-addition-collision} has a particularly simple interpretation: every collision created on the added line contributes exactly one unit to $\Delta$. Thus generic addition is the zero-collision case.

\begin{Corollary}
\label{cor:generic-line-bourbaki-addition}
Assume that $H$ contains no intersection point of $\A'$. Then
\[
e=e'+1,\qquad
\Bour(\A)=\Bour(\A')+e',
\qquad
\Delta(\A)=\Delta(\A').
\]
Moreover, for every added line, $\Bour(\A)-\Bour(\A')\leq e'$, with equality if and only if $H$ is generic with respect to $\A'$.
\end{Corollary}

\begin{proof}
Here $c_H=0$. Since $e'\leq n-2$, the stationary case in Proposition~\ref{prop:bourbaki-addition-collision} would require $0\geq n-e'-1>0$. Hence $e=e'+1$, and the assertions follow from the same proposition.
\end{proof}

We will also use the other natural defect-preserving addition.

\begin{Proposition}
\label{prop:addition-through-maximal-point}
Let $p\in L_2(\A')$ have maximal multiplicity $m=m_1(\A')$, and assume that $H$ passes through $p$ and through no other intersection point of $\A'$. Then $c_H=m-1$. Moreover:
\begin{enumerate}
\item if $e'<n-m$, then
\[
e=e'+1,\qquad
\Bour(\A)=\Bour(\A')+e'-m+1,
\qquad
\Delta(\A)=\Delta(\A')+m-1;
\]

\item if $e'=n-m$, then
\[
e=e',\qquad
\Bour(\A)=\Bour(\A'),
\qquad
\Delta(\A)=\Delta(\A').
\]
\end{enumerate}
\end{Proposition}

\begin{proof}
The $m$ lines through $p$ meet $H$ at $p$, while each of the remaining $n-m$ lines meets $H$ at a distinct point. Hence $c_H=m-1$. By \cite[Proposition~4.14]{Abe-Dimca-Stic}, one has $e=e'+1$ if $e'<n-m$ and $e=e'$ if $e'=n-m$. The formulas now follow from Proposition~\ref{prop:bourbaki-addition-collision}.
\end{proof}

The preceding two operations are in fact the only additions that preserve extremality.

\begin{Theorem}
\label{thm:extremal-addition-characterization}
Assume that $\A'$ is extremal. Then $\A=\A'\cup\{H\}$ is extremal if and only if exactly one of the following occurs:
\begin{enumerate}
\item $H$ contains no intersection point of $\A'$; in this case
$e=e'+1$;

\item $H$ passes through a point of maximal multiplicity
$m_1(\A')=n-e'$ and through no other intersection point of $\A'$; in this case $e=e'$.
\end{enumerate}
\end{Theorem}

\begin{proof}
Assume first that $\A$ is extremal. Since $\A'$ is also extremal, Proposition~\ref{prop:bourbaki-addition-collision} gives either $e=e'+1$ and $c_H=0$, or $e=e'$ and $c_H=n-e'-1$. The first condition says exactly that $H$ contains no intersection point of $\A'$.

Consider the second case. Since $\A'$ is extremal, Theorem~\ref{thm:extremal-triangular-arrangements} gives $m_2(\A')=2$ and $m_1(\A')=n-e'$ whenever $\A'$ is nonfree. If $\A'$ is free, then, since $\A'$ is extremal,
\[
0=\Delta(\A')=\binom{e'}2,
\]
and hence $e'=1$. The corresponding statement follows from the
classification of arrangements with initial degree one
\cite[Proposition~4.7]{Dimca-Ibadula-Macinic}. Thus in all cases $m_1(\A')=n-e'$ and $m_2(\A')=2$.

Since $c_H=m_1(\A')-1$ and $\A$ is again extremal, $H$ must pass through a point of maximal multiplicity and through no other intersection point of $\A'$. Indeed, passing through any other old intersection point would create a second point of multiplicity at least three.

Conversely, a generic addition preserves $\Delta$ by Corollary~\ref{cor:generic-line-bourbaki-addition}. If $H$ passes through a point of maximal multiplicity and no other intersection point, then $e'=n-m_1(\A')$, and Proposition~\ref{prop:addition-through-maximal-point} shows that $\Delta$ is again preserved. Since $\A'$ is extremal, so is $\A$.
\end{proof}

As a consequence, extremal arrangements admit a simple constructive description.

\begin{Corollary}
\label{cor:extremal-constructive-characterization}
An essential arrangement is extremal if and only if it can be obtained from a triangle by successive additions of the following two types:
\begin{enumerate}
\item a line containing no intersection point of the arrangement
constructed so far;

\item a line through a point of maximal multiplicity and through no
other intersection point of the arrangement constructed so far.
\end{enumerate}
\end{Corollary}

\begin{proof}
An arrangement obtained in this way is extremal by Theorem~\ref{thm:extremal-addition-characterization}.

Conversely, let $\A$ be extremal. Choose three nonconcurrent lines of $\A$ and order the remaining lines arbitrarily, obtaining a filtration
\[
\A_3\subset\A_4\subset\cdots\subset\A_{n+1}=\A.
\]
Every $\A_j$ is essential. By Proposition~\ref{prop:bourbaki-deletion-formula}, extremality is inherited by essential deletion, so every $\A_j$ is extremal. Applying Theorem~\ref{thm:extremal-addition-characterization} at each step gives the required description.
\end{proof}

\section{Combinatorial aspects of the Bourbaki degree}
\label{sec:combinatorics-bourbaki-degree}

We conclude by considering the dependence of the Bourbaki degree on the intersection lattice. Since the reduced characteristic polynomial is combinatorially determined, Theorem~\ref{thm:bourbaki-characteristic-polynomial} shows that, within a fixed lattice class, the Bourbaki degree can vary only through the initial degree of the syzygy module. This observation leads to a precise comparison, a sharp noncombinatoriality result, and a strong restriction on the exceptional case in which the initial degree varies while the Bourbaki degree remains constant.

\subsection{Combinatoriality and Ziegler pairs}
\label{subsec:combinatoriality-bourbaki}

We begin with a simple but useful identity.

\begin{Proposition}
\label{prop:bourbaki-same-lattice}
Let $\A$ and $\B$ be essential arrangements of $n+1$ lines with $L(\A)\simeq L(\B)$. Then
\[
\Bour(\A)-\Bour(\B)
=
(e_\A-e_\B)(e_\A+e_\B-n).
\]
Consequently, $\Bour(\A)=\Bour(\B)$ if and only if $e_\A=e_\B$ or $e_\A+e_\B=n$.
\end{Proposition}

\begin{proof}
The two arrangements have the same reduced characteristic polynomial, say
\[
\overline{\chi}_\A(t)
=
\overline{\chi}_\B(t)
=
t^2-nt+c.
\]
Hence, by Theorem~\ref{thm:bourbaki-characteristic-polynomial},
\[
\begin{aligned}
\Bour(\A)-\Bour(\B)
&=
\overline{\chi}_\A(e_\A)
-\overline{\chi}_\B(e_\B)\\ &= (e_\A-e_\B)(e_\A+e_\B-n).
\end{aligned}
\]
\end{proof}

Recent examples of Ziegler pairs show that the Bourbaki degree is not combinatorially determined. In fact, the first possible number of lines is sharp.

\begin{Theorem}
\label{thm:sharp-combinatorial-threshold}
Assume that $k=\CC$.
\begin{enumerate}
\item For essential arrangements with at most eight lines, the
Bourbaki degree is determined by the intersection lattice.

\item For every $q\geq0$, there exist essential arrangements
$\A_q$ and $\B_q$ of $9+q$ lines with $L(\A_q)\simeq L(\B_q)$ and
\[
e_{\A_q}=4+q,
\qquad
\Bour(\A_q)=2+4q+\binom q2,
\]
whereas
\[
e_{\B_q}=5+q,
\qquad
\Bour(\B_q)=3+5q+\binom q2.
\]
In particular,
\[
\Bour(\B_q)-\Bour(\A_q)=q+1.
\]
\end{enumerate}
Thus nine is the smallest number of lines for which the Bourbaki degree can fail to be combinatorially determined, and such examples exist for every number of lines at least nine.
\end{Theorem}

\begin{proof}
For arrangements with fewer than nine lines, Dimca and Pokora proved that the initial degree is determined by the intersection lattice; see \cite{Dimca-Pokora-Ziegler}. Since the reduced characteristic polynomial is also determined by the lattice, Theorem~\ref{thm:bourbaki-characteristic-polynomial} proves (1).

For (2), consider the nine-line pair constructed in \cite{Dimca-Pokora-nine}. The two arrangements have isomorphic intersection lattices, common weak combinatorics $(n_2,n_3,n_4)=(9,7,1),$ and initial degrees $4$ and $5$. Hence $b_2=9+2\cdot7+3=26,$ so their common reduced characteristic polynomial is $\overline{\chi}(t)=t^2-8t+18.$ Therefore their Bourbaki degrees are respectively
\[
\overline{\chi}(4)=2,
\qquad
\overline{\chi}(5)=3.
\]

Now add successively $q$ generic lines to each arrangement. Generic addition preserves the isomorphism type of the intersection lattice, and Corollary~\ref{cor:generic-line-bourbaki-addition} shows that at each step the initial degree increases by one and the Bourbaki degree increases by the preceding initial degree. Thus
\[
\Bour(\A_q)
=
2+\sum_{i=0}^{q-1}(4+i)
=
2+4q+\binom q2,
\]
and similarly
\[
\Bour(\B_q)
=
3+\sum_{i=0}^{q-1}(5+i)
=
3+5q+\binom q2.
\]
This proves the result.
\end{proof}

\subsection{Modularity defects}
\label{subsec:modularity-defects}

We next relate the Bourbaki degree to the failure of a point to be modular. For $p\in L_2(\A)$, define
\[
\delta_p(\A) :=
\sum_{\substack{
q\in L_2(\A)\setminus\{p\}\\
\overline{pq}\notin\A}}
(m_q-1),
\]
where $\overline{pq}$ denotes the line joining $p$ and $q$. Thus $\delta_p(\A)=0$ precisely when $p$ is modular.

The following factorization gives a direct relation between this combinatorial quantity and the Bourbaki degree.

\begin{Proposition}
\label{prop:characteristic-polynomial-modularity-defect}
Let $\A$ be an essential arrangement of $n+1$ lines and let $p\in L_2(\A)$ have multiplicity $m$. Then
\[
b_2(\A)
=
(m-1)+m(n+1-m)+\delta_p(\A),
\]
and
\begin{equation}
\label{eq:characteristic-modularity-defect}
\overline{\chi}_\A(t)
=
(t-m+1)\bigl(t-(n+1-m)\bigr)+\delta_p(\A).
\end{equation}
Consequently,
\begin{equation}
\label{eq:bourbaki-modularity-defect}
\Bour(\A)
=
(e_\A-m+1)\bigl(e_\A-(n+1-m)\bigr)+\delta_p(\A).
\end{equation}
\end{Proposition}

\begin{proof}
Let $\A_p$ be the set of the $m$ lines of $\A$ passing through $p$. If $q\neq p$ and $\overline{pq}\in\A$, then exactly one line through $q$ belongs to $\A_p$, while the remaining $m_q-1$ lines through $q$ do not pass through $p$.

Every pair consisting of a line through $p$ and a line not through $p$ meets at a unique point $q\neq p$ with $\overline{pq}\in\A$. Conversely, such a point $q$ accounts for exactly $m_q-1$ of these pairs. Hence
\[
\sum_{\substack{q\neq p\\ \overline{pq}\in\A}}
(m_q-1)
=
m(n+1-m).
\]
Separating the contribution of $p$, the points joined to $p$ by a line of $\A$, and the remaining points gives the formula for $b_2(\A)$. Since $\overline{\chi}_\A(t)=t^2-nt+b_2(\A)-n,$ we obtain \eqref{eq:characteristic-modularity-defect}. Evaluating at $t=e_\A$ and applying Theorem~\ref{thm:bourbaki-characteristic-polynomial} gives \eqref{eq:bourbaki-modularity-defect}.
\end{proof}

The most natural case is when the syzygy associated with $p$ has minimal degree.

\begin{Corollary}
\label{cor:bourbaki-modularity-defect}
Let $p\in L_2(\A)$ have multiplicity $m$ and assume that
\[
e_\A=n+1-m.
\]
Then $m=m_1(\A)$ and
\[
\Bour(\A)=\delta_p(\A).
\]
In particular:
\begin{enumerate}
\item if $\delta_p(\A)=0$, then $\A$ is free;
\item if $\delta_p(\A)=1$, then $\A$ is nearly free;
\item if $\delta_p(\A)=2$, then $\A$ is plus-one generated but not
nearly free.
\end{enumerate}
\end{Corollary}

\begin{proof}
The bound \eqref{eq:initial-degree-multiplicity-bound} gives $e_\A\leq n+1-m_1(\A)$. Since $e_\A=n+1-m$, one obtains $m\geq m_1(\A)$, and hence $m=m_1(\A)$. Formula \eqref{eq:bourbaki-modularity-defect} now gives $\Bour(\A)=\delta_p(\A)$.

The first two assertions follow from the characterizations of free and nearly free arrangements by Bourbaki degrees zero and one. If $\delta_p(\A)=2$, then $\Bour(\A)=2$. Since $\A$ is nonfree, $t(\A)\geq1$, while Proposition~\ref{prop:type-bourbaki-multiplicity-sandwich} gives $\binom{t(\A)+1}{2}\leq2.$ Thus $t(\A)=1$, so $\A$ is plus-one generated. Since its Bourbaki degree is not one, it is not nearly free.
\end{proof}

\subsection{The exceptional symmetric case}
\label{subsec:exceptional-symmetric-case}

Proposition~\ref{prop:bourbaki-same-lattice} shows that two realizations of the same intersection lattice can have different initial degrees but the same Bourbaki degree only when the two initial degrees are symmetric about $n/2$. The next theorem places strong restrictions on this exceptional case.

\begin{Theorem}
\label{thm:symmetric-bourbaki-obstruction}
Assume that $k=\CC$. Let $\A$ and $\B$ be essential arrangements of $n+1$ lines with $L(\A)\simeq L(\B)$, and assume that
\[
e_\A\neq e_\B,
\qquad
\Bour(\A)=\Bour(\B).
\]
Set
\[
q:=\min\{e_\A,e_\B\}, \qquad b:=\Bour(\A)=\Bour(\B),
\]
and write $m_1=m_1(\A)=m_1(\B)$. Then
\[
\{e_\A,e_\B\}=\{q,n-q\},
\]
both arrangements are nonfree, $b>0$, and
\[
m_1\leq q<n+1-m_1.
\]
Moreover,
\begin{equation}
\label{eq:symmetric-bourbaki-bounds}
\binom{n-2q+1}{2}
\leq b \leq
\binom q2-\binom{m_1-1}{2}.
\end{equation}
Consequently,
\begin{equation}
\label{eq:symmetric-bourbaki-obstruction}
\binom{n-2q+1}{2}
+
\binom{m_1-1}{2}
\leq
\binom q2,
\end{equation}
and
\[
\frac{n+1}{3}<q<\frac n2,
\qquad 2m_1<n.
\]
\end{Theorem}

\begin{proof}
By Proposition~\ref{prop:bourbaki-same-lattice}, $e_\A+e_\B=n$. After interchanging $\A$ and $\B$ if necessary, we may assume $e_\A=q, \qquad e_\B=n-q, \qquad q<n-q.$

If $b=0$, then both arrangements are free. Their common reduced characteristic polynomial would then have roots $q$ and $n-q$, and the initial degree of a free arrangement is the smaller root $q$. This contradicts $e_\B=n-q>q$. Hence $b>0$, and both arrangements are nonfree.

Applying \eqref{eq:initial-degree-multiplicity-bound} to $\B$ gives $n-q\leq n+1-m_1,$ so $m_1\leq q+1$. Moreover, $q\neq n+1-m_1$, since otherwise the same inequality would give $n-q\leq q$, a contradiction. Therefore the maximal-multiplicity estimate in Theorem~\ref{thm:upper-bound-line-arrangements} applies to $\A$ and gives $b\leq \binom q2-\binom{m_1-1}{2}.$ Since $b>0$, this also rules out $m_1=q+1$, and hence $m_1\leq q$. Together with \eqref{eq:initial-degree-multiplicity-bound} for $\A$ and the inequality $q\neq n+1-m_1$, this gives $m_1\leq q<n+1-m_1.$

Since $n-q>n/2$, Corollary~\ref{cor:large-e-bourbaki-window} applied to $\B$ gives
\[
b\geq
\binom{2(n-q)+1-n}{2}
=
\binom{n-2q+1}{2}.
\]
This proves \eqref{eq:symmetric-bourbaki-bounds} and \eqref{eq:symmetric-bourbaki-obstruction}.

It remains to determine the range of $q$. We already have $q<n/2$. Suppose that $q\leq(n+1)/3$. Then $n-2q+1\geq q$, and hence $\binom{n-2q+1}{2}\geq\binom q2.$ Together with \eqref{eq:symmetric-bourbaki-bounds}, this forces $m_1=2$ and $n+1=3q$. But $m_1=2$ means that all intersection points are double, so the arrangement is generic and has initial degree $n-1$, a contradiction. Thus $q>(n+1)/3$. Finally, $m_1\leq q<n/2$ gives $2m_1<n$.
\end{proof}

The first possible exceptional case is therefore very small.

\begin{Corollary}
\label{cor:first-symmetric-bourbaki-case}
Assume that $k=\CC$. If two essential arrangements with isomorphic intersection lattices have different initial degrees but the same Bourbaki degree, then they have at least ten lines.

If they have exactly ten lines, then their initial degrees are $4$ and $5$, and their common maximal intersection multiplicity is either $3$ or $4$. If $b$ denotes their common Bourbaki degree, then
\[
1\leq b\leq5 \quad\text{when }m_1=3, \qquad 1\leq b\leq3 \quad\text{when }m_1=4.
\]
\end{Corollary}

\begin{proof}
For arrangements with fewer than nine lines, the initial degree is determined by the intersection lattice \cite{Dimca-Pokora-Ziegler}. For nine lines, one has $n=8$, while Theorem~\ref{thm:symmetric-bourbaki-obstruction} would require an integer $q$ satisfying $3<q<4,$ which is impossible. Hence at least ten lines are necessary.

For ten lines, $n=9$, and the same theorem gives
\[
\frac{10}{3}<q<\frac92,
\]
so $q=4$ and the other initial degree is $5$. Moreover, $m_1\leq4$. The case $m_1=2$ is the generic lattice and cannot have two different initial degrees. Hence $m_1\in\{3,4\}$. The stated bounds for $b$ follow immediately from \eqref{eq:symmetric-bourbaki-bounds}.
\end{proof}

We can sharpen the ten-line case considerably. The following result shows that any exceptional pair with initial degrees $4$ and $5$ is forced into a very small number of weak-combinatorial types.

\begin{Proposition}
\label{prop:ten-line-symmetric-reduction}
Let $\A$ and $\B$ be essential arrangements of ten lines over $\CC$ such that $L(\A)\simeq L(\B)$, $e_\A=4$, and $e_\B=5$. Set $b:=\Bour(\A)=\Bour(\B)$. Then
\[
m_1(\A)=m_1(\B)=4,
\qquad b\in\{1,2\}.
\]
Moreover, for every $H\in\A$, one has $e_{\A\setminus\{H\}}=4$, and the collision number of $H$ relative to $\A\setminus\{H\}$ satisfies $c_H\geq4$.

If $n_j$ denotes the common number of $j$-fold points, then the only possible weak combinatorics are
\[
\begin{array}{c|c}
b & (n_2,n_3,n_4)\\ \hline 1 & (3,12,1),\ (6,9,2),\ (9,6,3),\ (12,3,4),\ (15,0,5),\\[1mm] 2 & (6,11,1),\ (9,8,2).
\end{array}
\]
In particular, when $b=1$ both arrangements are nearly free, whereas when $b=2$ both are plus-one generated but not nearly free.
\end{Proposition}

\begin{proof}
By Corollary~\ref{cor:first-symmetric-bourbaki-case}, $m_1(\A)=m_1(\B)\in\{3,4\}$. In particular, every deletion considered below is essential.

We first show that every deletion of $\A$ has initial degree $4$. Fix $H\in\A$, let $H'$ be the corresponding line of $\B$, and set $\A'=\A\setminus\{H\}$ and $\B'=\B\setminus\{H'\}$. By addition--deletion,
\[
e_{\A'}\in\{3,4\},
\qquad
e_{\B'}\in\{4,5\}.
\]
Suppose that $e_{\A'}=3$. If $\A'$ is free, then Corollary~\ref{cor:uniform-terao-triangular-criterion} applies, since $\A'$ has nine lines and $\binom42=6<9$. Hence $\B'$ is free with the same exponents, contradicting $e_{\B'}\geq4$.

If $\A'$ is not free, then Theorem~\ref{thm:low-degree-bourbaki-classification} gives $m_1(\A')=6$. Since $\A'$ and $\B'$ have the same intersection lattice, $m_1(\B')=6$, and \eqref{eq:initial-degree-multiplicity-bound} gives $e_{\B'}\leq9-6=3$, again a contradiction. Thus $e_{\A'}=4$ for every $H\in\A$.

Adding $H$ back is therefore a stationary addition. By Proposition~\ref{prop:bourbaki-addition-collision}, $c_H\geq9-4-1=4$ for every $H\in\A$.

Suppose now that $m_1=3$. Then only double and triple points occur. Since $\overline{\chi}_\A(4)=b$, the pair-counting identity gives $n_3=16-b, \qquad n_2=3(b-1).$ If a line contains $d_H$ double points and $t_H$ triple points, then $d_H+2t_H=9$, so $d_H$ is a positive odd integer. Hence $2n_2\geq10$, and therefore $b\geq3$.

On the other hand, in this case $c_H=t_H$, so
\[
3n_3=\sum_{H\in\A}c_H\geq40.
\]
Thus $3(16-b)\geq40$, which gives $b\leq2$, a contradiction. Therefore $m_1=4$.

Hence only double, triple and quadruple points occur, and the pair-counting identity together with $\overline{\chi}_\A(4)=b$ gives $n_3+3n_4=16-b, \qquad n_2=3(b-1+n_4).$ Moreover,
\[
\sum_{H\in\A}c_H
=
3n_3+8n_4
=
48-3b-n_4.
\]
Since $c_H\geq4$ for all ten lines,
\[
48-3b-n_4\geq40.
\]
As $n_4\geq1$, it follows that $b\leq2$. Since $b>0$, one has
$b\in\{1,2\}$.
Moreover, the same inequality gives
\[
n_4\leq8-3b.
\]
Together with
\[
n_3+3n_4=16-b,\qquad
n_2=3(b-1+n_4),
\]
this determines the possibilities. If $b=1$, then
$1\leq n_4\leq5$, and hence
\[
(n_2,n_3,n_4)
=
(3,12,1),\ (6,9,2),\ (9,6,3),\ (12,3,4),\ (15,0,5).
\]
If $b=2$, then $1\leq n_4\leq2$, and hence
\[
(n_2,n_3,n_4)
=
(6,11,1),\ (9,8,2).
\]

Finally, Bourbaki degree one characterizes nearly free arrangements.
If $b=2$, then
Proposition~\ref{prop:type-bourbaki-multiplicity-sandwich} gives
\[
\binom{t(\A)+1}{2}\leq2,
\qquad
\binom{t(\B)+1}{2}\leq2.
\]
Since both arrangements are nonfree, their types are positive, and
therefore $t(\A)=t(\B)=1$. Hence both arrangements are plus-one
generated. Since their Bourbaki degree is not one, they are not
nearly free.
\end{proof}

\begin{Corollary}
\label{cor:ten-line-symmetric-deletion-dichotomy}
Under the hypotheses of Proposition~\ref{prop:ten-line-symmetric-reduction}, exactly one of the following occurs:
\begin{enumerate}
\item there exists a line $H$ with $c_H>4$; for every such line, the
corresponding deletions $\A\setminus\{H\}$ and $\B\setminus\{H'\}$ have initial degrees $4$ and $5$, respectively;

\item $c_H=4$ for every line $H$, and the weak combinatorics is
$(n_2,n_3,n_4)=(15,0,5)$ with $b=1$, or $(n_2,n_3,n_4)=(9,8,2)$ with $b=2$.
\end{enumerate}
\end{Corollary}

\begin{proof}
Assume that $c_H>4$ for some $H\in\A$. We already know that $e_{\A\setminus\{H\}}=4$. Let $H'$ be the corresponding line of $\B$. Since the lattice isomorphism preserves the incidences on $H$ and $H'$, their collision numbers are equal. If $e_{\B\setminus\{H'\}}=4$, then adding $H'$ produces a jump from $4$ to $5$, and Proposition~\ref{prop:bourbaki-addition-collision} gives $c_H\leq4$, a contradiction. Hence $e_{\B\setminus\{H'\}}=5$, proving (1). Otherwise $c_H=4$ for every $H$, so $3n_3+8n_4=40.$ Together with $n_3+3n_4=16-b$, this gives $n_4=8-3b$. For $b=1$ one obtains $(n_2,n_3,n_4)=(15,0,5)$, while for $b=2$ one obtains $(n_2,n_3,n_4)=(9,8,2)$.
\end{proof}

\begin{Question}
\label{prob:ten-line-symmetric-bourbaki}
Does there exist a pair of essential ten-line arrangements over $\CC$ with isomorphic intersection lattices and initial degrees $4$ and $5$?

By Proposition~\ref{prop:ten-line-symmetric-reduction} and Corollary~\ref{cor:ten-line-symmetric-deletion-dichotomy}, any such pair either arises from a nine-line pair with initial degrees $4$ and $5$ by a suitable addition, or has weak combinatorics
\[
(15,0,5) \qquad\text{or}\qquad (9,8,2).
\]
\end{Question}
Thus noncombinatoriality of the Bourbaki degree already occurs for nine lines. By contrast, the exceptional phenomenon in which the initial degree changes while the Bourbaki degree remains constant can occur, if at all, only from ten lines onward. In the ten-line case, Proposition~\ref{prop:ten-line-symmetric-reduction} and Corollary~\ref{cor:ten-line-symmetric-deletion-dichotomy} reduce the problem to a small number of explicit combinatorial possibilities.


\end{document}